\documentclass[12pt]{amsart}
\usepackage{preamble}

\begin{document}

\title{
Path-Kernel Method for Differentiating Unstable Diffusions
}

\begin{abstract}
We derive and rigorously prove the path-kernel formula for the linear response of stochastic differential equations, namely, the derivative of averaged statistics with respect to a model parameter.
The parameter may affect the diffusion coefficient, as well as the drift coefficient and initial condition.
No hyperbolicity assumption is required.
The proof introduces a single change of variables for the state path that transfers part of the unstable path perturbation to differentiation of the probability kernel.
We also derive a pathwise Monte Carlo algorithm for estimating the response and demonstrate it on the 40-dimensional noisy Lorenz--96 system.

\smallskip
\noindent \textbf{AMS subject classification numbers.}
60H07, 
65D25, 
37M25, 
65C05, 

\smallskip
\noindent \textbf{Keywords.}
Linear response,
chaos,
SDEs,
likelihood-ratio,
path-perturbation.

\end{abstract}

\maketitle

\section{Introduction}
\label{s:intro}

Many applications of stochastic models concern averaged statistics, such as expected observables at a fixed time or long-time averages, rather than individual trajectories.
The derivative of such a statistic with respect to a model parameter is called its linear response.
The parameter considered here may control the initial condition, drift coefficient, or diffusion coefficient.
Linear response is used in sensitivity analysis and gradient-based optimization, including parameter estimation in data assimilation.
Accurately estimating it from simulated trajectories is an important challenge in computational statistics, especially for high-dimensional and unstable systems.

\subsection{Main results}
\label{s:mainResults}

This paper derives the path-kernel formula for the linear response of random dynamical systems in \Cref{l:Dt}, providing a basic tool for the numerical optimization of diffusion processes.
We prove the finite-time formula rigorously for both discrete-time systems and continuous-time stochastic differential equations (SDEs), and then formally pass to the infinite-time limit.
Previous linear response tools had only limited applicability, since they typically require purely stable or purely unstable dynamics, very low-dimensional systems, strong hyperbolicity, or fixed diffusion coefficients; none of these assumptions holds in typical applications.
Our new tool breaks these limitations and therefore makes linear response methods available for a broader class of applied problems.
For example, it enables solving data assimilation problems in much greater generality (high-dimensional, chaotic, non-synchronized systems with unknown parameters and partial and noisy observations), as shown in a follow-up work \cite{apk}.

In the main theorem below, $B_t$ denotes a standard \(M\)-dimensional Brownian motion.
To keep the notation from becoming too heavy, we omit the subscript \(t\) when no confusion arises.
Let \(\gamma^*\) be the parameter value under consideration.
Let $\gamma$ be the parameter controlling the dynamics and the initial condition, and hence the distribution of the process $\{X_t^\gamma\}_{t\ge0}$; by default, \(\gamma=\gamma^*\), so we write \(X:=X^{\gamma^*}\).
Let $F^\gamma(\cdot)$ be a vector field, $\sigma^\gamma(\cdot)$ be a scalar function, and define the spatial derivative by $\nabla_v F:=\pp F x v$.
We use \(\delta\) to denote differentiation with respect to \(\gamma\) at \(\gamma=\gamma^*\).
Let $\Phi$ be a fixed observable function.
The details of notations and assumptions are in \Cref{s:notation}.

\begin{restatable}[path-kernel formula for differentiating SDEs]{theorem}{goldbach}
\label{t:main}
Fix any \(x_0\), \(v_0\), and any scalar adapted process \(\alpha_t\) (called a `schedule') independent of \(\gamma\).
Consider the Ito SDE,
\[
dX^\gamma_t
= F^\gamma(X^\gamma_t) dt 
+ \sigma^\gamma(X^\gamma_t)dB_t,
\quad
X_0^\gamma = x_0^\gamma :=  x_0 + (\gamma-\gamma^*) v_0.
\]
Note that $F^\gamma(\cdot)$ and $\sigma^\gamma(\cdot)$ depend on the parameter $\gamma$.
Let $v_t$ be the solution of the damped path-perturbation equation from $v_0$,
\[
dv
= - \alpha_t v dt 
+ \left(\nabla_v F(X) + \delta F^\gamma (X) \right) dt
+ \left(\nabla_v \sigma(X) + \delta \sigma^\gamma (X)\right)dB.
\]
Under \Cref{ass}, the linear response has the expression
\[
\delta \E{\Phi(X^\gamma_T)}
=
\E{ \nabla\Phi(X_T) \cdot v_T
+\Phi(X_T) \int_{t=0}^{T} \frac{\alpha_t v_t }{\sigma(X_t)}\cdot dB_t}.
\]
\end{restatable}

The main difficulties of this problem are as follows.
First, the system may be unstable, so the path perturbation can grow very rapidly, which causes serious convergence difficulties when \(T\) is large.
Second, when the diffusion coefficient is parameterized, the governing equation for \(v\) contains the inhomogeneous stochastic term \(\delta \sigma\, dB\), so we cannot apply the likelihood-ratio method.
In our new formula, we typically choose \(\alpha_t \ge 0\), so that the term \(\alpha_t v_t\,dt\) damps the unstable growth of the path perturbation \(v_t\).
This term represents the portion of the path perturbation that is shifted to the probability kernel.
\Cref{t:ergodic} formally presents the linear response formula along a single infinite-time orbit for the stationary measure, or equivalently for the long-time averaged statistic.

\subsection{Literature review}
\label{s:review}

There are three basic methods for expressing and computing linear response: the path-perturbation method (shortened to the path method), the divergence method, and the kernel-differentiation method (shortened to the kernel method).

The first basic method is the path-perturbation method, which computes the linear response by averaging the path perturbation over many sample paths. It is also known in various contexts as the ensemble method, stochastic gradient method, backpropagation, adjoint method, or a Koopman-operator method \cite{eyink2004ruelle,lucarini_linear_response_climate,Caflisch2021}.
Proofs of this formula for the physical (SRB) measure of deterministic hyperbolic systems were given, for example, in \cite{Dolgopyat2004,Ruelle_diff_maps,Jiang2012}.
This framework also includes backpropagation, which is the basic algorithm underlying much of machine learning.
However, when the system is chaotic, so that the pathwise perturbation grows exponentially fast in time, the method becomes prohibitively expensive, since the variance of the estimator also grows exponentially.
A common workaround is to artificially reduce the size of the path perturbation. One may minimize the path perturbation, leading to shadowing methods \cite{Ni_nilsas,Ni_NILSS_JCP}.
One may also explicitly introduce a damping (or clipping) term \cite{Talnikar_thesis,clip_gradients2}, or damp only selected frequencies when the unstable directions are concentrated in frequency \cite{FP25}.
However, it was previously difficult to compute the additional error introduced by such damping; this paper provides the answer for random dynamical systems.

The second basic linear response method is the divergence method, also known as the transfer-operator method.
The perturbation of the measure transfer operator takes the form of a divergence.
For deterministic systems, the traditional formula may not be defined pointwise in contracting directions, which prevents direct sampling and makes high-dimensional computation expensive \cite{Gouezel2008,Baladi2017,Galatolo2014,Wormell2019a}.
For SDEs, differentiating the Fokker--Planck equation leads to an analogous divergence formula, but its second-order term cannot be represented by a single path \cite{Lucarini2026}.

For deterministic hyperbolic systems, the fast response formula combines the path-perturbation and divergence methods into a pointwise, sampleable formula without distributions or exponentially growing terms \cite{Ni_asl,TrsfOprt,vdivF,fr,Ni_nilsas,far}.

The third basic linear response method is the kernel-differentiation method, which applies only to random systems.
For SDEs, it follows directly from the Girsanov transformation \cite{CM44}, which underlies a major part of Malliavin calculus \cite{MalliavinBook}.
It is also known as the likelihood-ratio method or the Monte Carlo gradient method \cite{Rubinstein1989,Reiman1989,Glynn1990}.
Proofs of the kernel-differentiation method for stationary measures were given in \cite{HaMa10,bahsoun20,GG19}.
For random dynamics, this method is more robust than the previous two, since it does not involve Jacobian matrices \cite{sedro23}.

In numerical computation, the kernel-differentiation method naturally allows Monte Carlo computation and is therefore efficient in high dimensions.
In \cite{Ni_kd}, we gave an ergodic version of the kernel-differentiation method for stationary measures, meaning that it can be implemented along a single orbit.
We also showed that the method remains valid, and can be significantly faster, when the noise and perturbations are along a foliation of submanifolds \cite{Ni_kd}.
A major shortcoming of the kernel-differentiation method is that its variance is inversely proportional to the scale of the noise, so the computational cost becomes high when the noise is small.
A possible remedy is the triad program proposed in \cite{Ni_kd}, which requires progress in all three methods.

One of the most important open problems for the kernel-differentiation method for SDEs is that previous works do not allow the diffusion coefficient to be parameterized or perturbed \cite{PSW23,ZD21}.
As we will see in \Cref{s:deker}, this is essentially another manifestation of the main shortcoming of kernel methods: the noise is relatively small compared to the perturbation.
\Cref{s:degauss} gives a simple example illustrating why the path-kernel method resolves this problem.
The Bismut formula \cite{Bismut84,EL94,PW19bismut,HM06} provides a related idea, but it applies only to derivatives with respect to the initial condition.
\Cref{s:debis} explains how our formula reduces to the Bismut formula in this case.
Even in that degenerate case, our main proof uses a single state-path change of variables, without first decomposing the tangent or constructing a compensating noise control.

The path-kernel method in this paper unifies the path-perturbation method and the kernel-differentiation method.
It provides a gradient tool for the optimization of diffusion processes, especially in unstable settings.
In a follow-up paper \cite{apk}, we give the adjoint path-kernel formula and use it to solve the \textit{data assimilation} problem in a general setting with chaotic high-dimensional dynamics, partial and noisy observations, unknown parameters, a single loss over a long time window, and no synchronization assumption.

In related work, the divergence-kernel method gives pathwise formulas of Feynman--Kac type for the score and density response of SDEs and supports a generative-modeling framework based on minimizing the KL divergence to data \cite{divKer,DKlinR}.
In contrast, the present path-kernel method differentiates averaged observables and permits singular initial conditions, making it more directly applicable to data assimilation.
The costs of the two methods depend on opposite ends of the Lyapunov spectrum, making them complementary components of the triad program proposed in \cite{Ni_kd}.

\subsection{Structure of the paper}
\label{s:structure}

\Cref{s:derive} derives and proves the new formula.
\Cref{s:intui} explains the main intuition behind the proof, namely, the gradual transfer of the path perturbation to the probability kernel.
It also gives an intuitive proof based on directly comparing bundles of paths across different parameter values.
\Cref{s:discrete} derives the formula for discrete-time random dynamical systems.
\Cref{s:cts} proves the convergence to the continuous-time limit for SDEs.
\Cref{s:infinite} then formally passes to the infinite-time limit.

\Cref{s:numeric} discusses practical issues arising in applications.
\Cref{s:howtouse} gives a rough estimate of the variance of the estimator when \(\alpha\) is a constant, and explains how to choose \(\alpha\).
\Cref{s:procedure} provides a detailed description of the algorithm.
\Cref{s:example} applies the formal ergodic version of the algorithm to the noisy Lorenz--96 system. Its deterministic part appears not to admit a linear response, but after adding noise, the path-kernel algorithm reasonably reflects the relationship between the long-time averaged observable and the parameter.
This problem cannot be handled by previous methods, since the system is unstable and the diffusion coefficient is also parameterized.

The appendices provide supporting results.
\Cref{s:removeTrunc} removes the invertibility assumption on the path map using a technical argument based on truncation.
\Cref{s:proof2} gives an alternative derivation based on the Girsanov theorem, and \Cref{s:degen} shows how the path-kernel formula reduces to several well-known formulas.

In a follow-up paper \cite{apk}, we give the adjoint path-kernel formula and use it to solve the data assimilation problem in a general setting.
In another recent line of work, we derived the divergence-kernel formula for the score and for the linear response of the marginal density of SDEs, which significantly broadens the class of models available for generative modeling \cite{divKer,DKlinR}.

\section{Deriving and proving the formula}
\label{s:derive}

\subsection{Notations and assumptions}
\label{s:notation}

In the Euclidean space \(\R^M\), 
consider the SDE
\[
dX_t^\gamma = F^\gamma(X_t^\gamma)\,dt + \sigma^\gamma(X_t^\gamma)\,dB_t,
\quad
X_0^\gamma = x_0 + (\gamma-\gamma^*) v_0.
\]
Here $F^\gamma:\R^M\to \R^M$ is a vector field, and $\sigma^\gamma:\R^M\to \R$ is a scalar-valued function;
\(B\) denotes a standard \(M\)-dimensional Brownian motion, \(x_0\) is the deterministic initial condition, and \(v_0\) is the perturbation direction of the initial condition.
All SDEs in this paper are understood in the It\^o sense, so their time-discretized version takes the form given in \Cref{e:discreteSDE}.

Here \(\gamma\) is the parameter controlling the dynamics and the initial condition.
Let \(\gamma^*\in I\) be the parameter value under consideration.
By default, we work at \(\gamma=\gamma^*\) and suppress this superscript for baseline objects; for example, we write \(X_t:=X_t^{\gamma^*}\), \(F:=F^{\gamma^*}\), and \(\sigma:=\sigma^{\gamma^*}\).
We use \(\delta\) for differentiation with respect to the parameter at \(\gamma=\gamma^*\).
Thus, for any parameter-dependent quantity \(A^\gamma\),
\[
\delta A
:=
\left.\frac{d}{d\gamma}\right|_{\gamma=\gamma^*} A^\gamma.
\]
When writing expressions such as \(\delta F^\gamma\) and \(\delta\sigma^\gamma\), the superscript indicates the parameterized family being differentiated; the derivative is evaluated at \(\gamma=\gamma^*\).
We use $\nabla$ to denote spatial derivatives, and
\begin{equation*}\begin{split}
  \nabla_v F := \nabla F \, v = \pp F x v.
\end{split}\end{equation*}
Hence, $\nabla_vF(x)$ is the Riemannian derivative of the vector field $F$ along the direction $v$, evaluated at $x$ (note that $v$ must be a vector at $x$).
Let \(\Phi\) be a fixed observable function.
Given a finite time interval \([0,T]\), our goal is to derive a pathwise expression for \(\delta \E{\Phi(X_T^\gamma)}\).
Let \(\cF_t\) be the \(\sigma\)-algebra generated by 
\(\{B_\tau\}_{\tau\le t}\) 
and \(X_0\) (we could also let \(X_0\) be distributed according to a given measure, without changing anything essential).
We take \(\alpha_t\) to be a scalar process, which we call the `schedule'.

We omit the subscript \(t\) when no confusion arises.
For random variables that may also appear as dummy variables in integrals, such as \(X\) and \(B\), we use uppercase letters for the random variables and the corresponding lowercase letters ($x$ and $b$) for values they may take or for dummy variables.
For functions of such variables, such as \(v_t\) and \(\alpha_t\), we do not change the letter case.

We now state some technical assumptions.
These assumptions can be relaxed; for instance, the boundedness and Lipschitz conditions on \(\alpha_t\) can be significantly weakened.
However, in this paper we restrict ourselves to proving the result in a simpler setting.

\begin{assumption} \label{ass}
For a parameter interval $I=[\gamma_{\min},\gamma_{\max}]$ containing $\gamma^*$,
assume that there exist constants $L, K, \sigma^*, \alpha^*, L_\alpha, L'>0$ such that
the partial derivatives $\nabla F^\gamma$, $\nabla \sigma^\gamma$,
$\delta F^\gamma$, $\delta \sigma^\gamma$, and $\nabla \Phi$ exist.
For \(\alpha\in[-\alpha^*,\alpha^*]\), define
\[
G_\alpha^\gamma(x,v)
:= -\alpha v+\nabla_v F^\gamma(x)+\delta F^\gamma(x),
\qquad
H^\gamma(x,v)
:= \nabla_v\sigma^\gamma(x)+\delta\sigma^\gamma(x).
\]
With this notation, the joint equations for \(X^\gamma\) and \(v^\gamma\) are
\[
\begin{aligned}
dX_t^\gamma
&=F^\gamma(X_t^\gamma)\,dt+\sigma^\gamma(X_t^\gamma)\,dB_t,\\
dv_t^\gamma
&=G_{\alpha_t}^\gamma(X_t^\gamma,v_t^\gamma)\,dt
+H^\gamma(X_t^\gamma,v_t^\gamma)\,dB_t.
\end{aligned}
\]
For all $\gamma\in I$, $t, s \in [0,T]$, and $x, y\in\R^M$,
\begin{enumerate}[label=(A\arabic*), labelindent=0pt, leftmargin=*, align=left]
\item (Uniform bounds for the coefficients and their derivatives)
\[
\begin{aligned}
&|\nabla F^\gamma(x)|,\,
|\nabla\sigma^\gamma(x)|\le L,\\
&|\nabla F^\gamma(x)-\nabla F^\gamma(y)|,\,
|\nabla\sigma^\gamma(x)-\nabla\sigma^\gamma(y)|,\\
&|\delta F^\gamma(x)-\delta F^\gamma(y)|,\,
|\delta\sigma^\gamma(x)-\delta\sigma^\gamma(y)|
\le L|x-y|.
\end{aligned}
\]
\item (Uniform linear growth)
\[
|F^\gamma(x)|^2+|\sigma^\gamma(x)|^2
+|\delta F^\gamma(x)|^2+|\delta\sigma^\gamma(x)|^2
\le K(1+|x|^2).
\]
\item $\alpha$ is \(\cF_t\)-adapted with a.s. Lipschitz and bounded paths, and for each time partition used below, \(\alpha_{n\Delta t}\) is determined by the corresponding discrete state history \(\{X_i\}_{i\le n}\),
\[
|\alpha_t| \le \alpha^* < \infty ,
\quad \textnormal{} \quad
|\alpha_t - \alpha_s| \le L_\alpha |t-s|.
\]
\item (Uniform positivity of $\sigma$)
\[
\sigma^\gamma(x) \ge \sigma^* > 0.
\]
\item (Lipschitz for $\Phi, \nabla \Phi$)
\[
|\Phi(x)-\Phi(y)| ,\, |\nabla \Phi(x) - \nabla \Phi(y)| \le L'|x-y|.
\]
\end{enumerate}
\end{assumption}

We perform our derivation after dividing the time interval into small segments of length $\Dt$.
Let $N$ be the total number of segments, so $N\Dt = T$.
Denote
\[ 
  \DB_n:= B_{(n+1)\Dt} - B_{n\Dt} \in \R^M,
  \quad
  \DB_n \sim N(0,\Dt I_M).
\]
Denote $\alpha_n = \alpha_{n\Dt}$.
The discretized SDE is
\begin{equation} \label{e:discreteSDE}
X^\gamma_{n+1} - X^\gamma_{n}
=F^\gamma(X^\gamma_n) \Dt + \sigma^\gamma(X^\gamma_n)\DB_n,
\quad
X_0^\gamma = x_0 + (\gamma-\gamma^*) v_0.
\end{equation}


\subsection{Intuition}
\label{s:intui}

We now give some intuitive ideas behind the derivation, which will be proved rigorously in later sections.
Our derivation is elementary and is based on directly comparing paths.
More specifically, we compare the location, probability density, and volume of bundles of paths across different values of \(\gamma\).

First, partition the time interval into small segments of length \(\Dt\).
Our goal is to understand what contributes to \(\E{\Phi(X_N^\gamma)} - \E{\Phi(X_N)}\) and then compute it.
For an SDE, there are many possible paths starting from \(x_0\) and \(x_0^\gamma\), so we need a rule for pairing paths from the two dynamics in order to compare them path by path and obtain the difference in \(\Phi\) at the final step.

To do this, we express \(x_t^\gamma\) in terms of \(x_t\) over the entire time interval by writing
\[
x_t^\gamma = x_t + (\gamma-\gamma^*) v_t,
\]
where \(v_t\) depends on \(v_0\) and the previous values of \(x\).
Since a single path has zero probability, we instead compare bundles of paths: a bundle under \(\gamma=\gamma^*\) and the corresponding bundle under a nearby parameter value \(\gamma\).
In \Cref{f:intuition}, these two bundles are illustrated in blue and red, respectively.

\begin{figure}[ht]
\centering
  \includegraphics[width=0.45\textwidth]{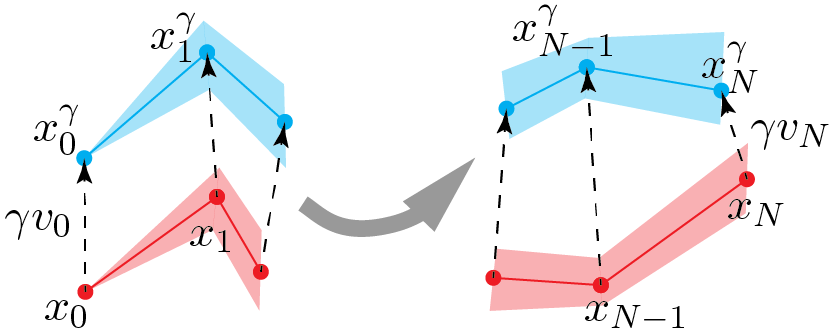}
  \caption{Intuition behind the derivative. For the two different values of \(\gamma\), we compare the blue bundle with the red bundle.}
  \label{f:intuition}
\end{figure}

There are three factors in the comparison between the red bundle and the blue bundle in \Cref{f:intuition}.
The first factor is that \(x_N^\gamma \neq x_N\), which induces a difference in \(\Phi\).
This is the path-perturbation contribution, represented in the next subsection by the term \(\nabla\Phi(x_N)\cdot v_N\).

Next, we must take into account that the probability of obtaining the blue bundle is different from that of obtaining the red bundle.
This is captured by differentiating the probability kernel, and it further splits into two factors.
The second factor is that the probability density assigned to each path is different.
This corresponds to the term \(\delta p(x_n^\gamma,x_{n+1}^\gamma)\) in the next subsection.

The third factor is that the corresponding bundles do not occupy the same Lebesgue volume.
Once we specify the correspondence between the sequences \(\{x_n^\gamma\}_{n\ge 0}\) and \(\{x_n\}_{n\ge 0}\), we also specify a correspondence between a region in path space for \(\{x_n^\gamma\}_{n\ge 0}\) and a region in path space for \(\{x_n\}_{n\ge 0}\).
These two regions may have different volumes, and this also contributes to the total probability they carry.
This is the term \(\delta J_{n+1}^\gamma\) in the next subsection.

How should we choose the correspondence between \(x_n^\gamma\) and \(x_n\)?
A first natural guess is to fix a background Brownian motion and let \(x_n^\gamma\) and \(x_n\) be generated by the same Brownian increments, but under different dynamics.
That is, for the same \(\Delta b_n\), let
\begin{equation*}\begin{split}
x_{n+1} &= x_n + F(x_n)\Dt + \sigma(x_n)\Delta b_n, \\
x_{n+1}^{\prime\gamma} &= x_n^{\prime\gamma} + F^\gamma(x_n^{\prime\gamma})\Dt + \sigma^\gamma(x_n^{\prime\gamma})\Delta b_n .
\end{split}\end{equation*}
Ignoring higher-order terms, this is equivalent to writing
\begin{equation*}\begin{split}
x_{n+1}^{\prime\gamma} \approx x_{n+1} + (\gamma-\gamma^*) u_{n+1},
\end{split}\end{equation*}
where \(u_n\) is the pure path perturbation, defined inductively by
\begin{equation*}\begin{split}
u_{n+1}
= u_n + \nabla_{u_n}F(x_n)\Dt + \delta F^\gamma(x_n)\Dt
+ \nabla_{u_n}\sigma(x_n)\Delta b_n + \delta \sigma^\gamma(x_n)\Delta b_n.
\end{split}\end{equation*}
Because the same \(\Delta b_n\) is used for both \(x_{n+1}\) and \(x_{n+1}^{\prime\gamma}\), the two paths have the same probability. 
Hence, the total contribution of the second and third factors vanishes, although they are not individually zero.
Thus we only need to compute the first factor.
This is the pure path-perturbation method, which does not work well for unstable systems; 
see \Cref{s:depath} for how our result degenerates to the pure path-perturbation formula.

A second guess is to simply set \(x_n^\gamma \equiv x_n\).
Then the first and third factors vanish, and only the change in the probability density remains.
This is the pure kernel-differentiation method, which does not allow parameterized diffusion;
see \Cref{s:deker} for how our result degenerates to the pure kernel-differentiation formula.

The main idea of this paper is to define \(x_{n+1}^\gamma\) by
\begin{equation*}\begin{split}
x^\gamma_{n+1}
=
x^\gamma_n
-\alpha_n(x_n^\gamma-x_n)\Dt
+F^\gamma(x_n^\gamma)\Dt
+\sigma^\gamma(x_n^\gamma)\Delta b_n,
\end{split}\end{equation*}
where the term with \(\alpha_n>0\) pulls \(x^\gamma_{n+1}\) toward \(x_{n+1}\), and hence damps the divergence between the two paths.
At the same time, we still want \(\{x_n^\gamma\}_{n\ge0}\) to be a path of the \(\gamma\)-dynamics.
To achieve this, the additional damping term must be absorbed into the Brownian motion.
Ignoring higher-order terms, the corresponding Brownian increment is
\begin{equation*}\begin{split}
\Delta b_n^\gamma
\approx
\Delta b_n-\frac{\alpha_n}{\sigma(x_n)}(x_n^\gamma-x_n)\Dt.
\end{split}\end{equation*}
Hence \(x_{n+1}^\gamma\) is still generated by the \(\gamma\)-dynamics, but with a modulated Brownian motion:
\begin{equation*}\begin{split}
x^\gamma_{n+1}
=
x^\gamma_n+F^\gamma(x_n^\gamma)\Dt+\sigma^\gamma(x_n^\gamma)\Delta b_n^\gamma.
\end{split}\end{equation*}
Since we use different Brownian motions, the probabilities of obtaining the two paths are different, and this must be accounted for by the second and third factors.

Our definition of \(x_n^\gamma\) is equivalent to
\begin{equation*}\begin{split}
x^\gamma_{n+1}
&\approx x_{n+1}+(\gamma-\gamma^*) v_{n+1},\\
v_{n+1}
&=
v_n-\alpha_n v_n\Dt+\nabla_{v_n}F(x_n)\Dt+\delta F^\gamma(x_n)\Dt
+\nabla_{v_n}\sigma(x_n)\Delta b_n+\delta\sigma^\gamma(x_n)\Delta b_n.
\end{split}\end{equation*}
Thus we subtract a small portion, \(\alpha_n v_n\Dt\), from the path-perturbation equation.
This portion is no longer propagated to later times; instead, it is transferred to kernel differentiation.
The effect of the portion shifted to kernel differentiation will be given in \Cref{e:shift}.
On the other hand, the remaining part of the path perturbation is propagated to the next step, and part of it may still reach step \(N\) and contribute to the perturbation of \(\Phi(X_N)\).
This tempered path perturbation gives our rule for pathwise comparison.

Finally, we remark that our derivation and proof make the formula easier to understand and discover.
This proof has a structure similar to that of our equivariant divergence formula for deterministic hyperbolic systems \cite{TrsfOprt,vdivF}, and is therefore useful for future developments such as a unification with the divergence method.
\Cref{s:proof2} gives a sketch of a proof based on the Girsanov theorem, which is shorter to write but much less intuitive.

\subsection{Deriving the discrete-time formula}
\label{s:discrete}

\subsubsection{Changing variables}
\label{s:changevar}

Partition the interval into small segments of length \(\Dt\), so that the dynamics are governed by \Cref{e:discreteSDE}.
The main result of this subsubsection is \Cref{e:ephi2}, which rewrites the perturbed expectation in terms of the unperturbed path variables.
We obtain it by changing the dummy variables from the perturbed path to the unperturbed path according to a prescribed correspondence between paths for different values of \(\gamma\).

After partitioning the time interval, \(\E{\Phi(X_N^\gamma)}\) has the expression
\begin{equation} \label{e:ephi}
\E{\Phi(X^\gamma_N)} =
\int_{x_1}\cdots\int_{x_{N}}
\Phi(x^\gamma_N)\, p(x^\gamma_0,x^\gamma_1)\cdots p(x^\gamma_{N-1},x^\gamma_N)\,
dx^\gamma_1\cdots dx^\gamma_N.
\end{equation}
Note that the dummy variables here are the path \(\{x_n^\gamma\}_{n=0}^N\), not the Brownian increments \(\{\Db_n\}_{n=0}^{N-1}\).
Here \(p(x_n^\gamma,x_{n+1}^\gamma)\) denotes the density of \(X_{n+1}^\gamma\) conditioned on \(X_n^\gamma\); it is a Gaussian function given by
\[
p(x^\gamma_n,x^\gamma_{n+1})
= (2\pi\Delta t)^{-\frac M2}
(\sigma^\gamma)^{-M}(x^\gamma_n)
\exp{-\frac{|x^\gamma_{n+1} - x^\gamma_n - F^\gamma(x^\gamma_n) \Delta t|^2 }{2(\sigma^\gamma)^2(x^\gamma_n)\Dt}}.
\]

For each \(n\), we formally change variables from the perturbed coordinate \(x_{n+1}^\gamma\) to the unperturbed coordinate \(x_{n+1}\) using the path correspondence below. Since the associated map need not be globally invertible, \Cref{s:removeTrunc} gives a rigorous proof based on a truncated map.
\begin{equation}\label{e:changeVar}
  x_{n+1}^\gamma = x_{n+1} + (\gamma-\gamma^*) v_{n+1},
\end{equation}
where \(v_n\), starting from \(v_0\), solves
\begin{equation} \label{e:vDt}
v_{n+1} 
= v_n - \alpha_n v_n \Dt
+ \delta F^\gamma (x^\gamma_n) \Dt
+ \delta\sigma^\gamma(x^\gamma_n) \Db_n.
\end{equation}
Here the total derivatives are
\[
\delta F^\gamma (x^\gamma_n) := \nabla_{v_n}F(x_n) + \delta F^\gamma(x_n),
\quad
\delta \sigma^\gamma (x^\gamma_n) := \nabla_{v_n}\sigma(x_n) + \delta \sigma^\gamma (x_n).
\]

We change variables sequentially. When changing the variable \(x_{n+1}^\gamma\) to \(x_{n+1}\), the variables \(\{x_i^\gamma\}_{i\le n}\) have already been changed.
Hence, for the \((n+1)\)-st change of variable, the Jacobian is
\[
J_{n+1}^\gamma = \det \pp{x_{n+1}^\gamma}{x_{n+1}},
\]
with \(\{x_i^\gamma\}_{i\le n}\) fixed in the partial derivative matrix.
Therefore, we need to determine how \(x_{n+1}\) appears in the expression of \(x_{n+1}^\gamma\) when the latter is viewed as a function of the unperturbed coordinates \(\{x_i\}_{i\le n+1}\).

Due to the definition of \(x_{n+1}^\gamma\), we need to express \(v_{n+1}\) in the unperturbed coordinates.
In the expression for \(v_{n+1}\), only \(\Db_n\) depends on \(x_{n+1}\), whereas all other terms, including \(\alpha_n\), have already been determined by \(\{x_i\}_{i\le n}\) (here we use the discrete-history condition in (A3)).
So it remains to express \(\Db_n\) in the unperturbed coordinates:
\begin{equation}\label{e:db}
\Db_n = \frac 1{\sigma(x_n)} \left(x_{n+1}-x_n - F(x_n)\Dt\right).
\end{equation}
Summarizing, the change-of-variable formula is
\[
x_{n+1}^\gamma = x_{n+1} \left( 1+  \frac {\gamma-\gamma^*} {\sigma(x_n)} \delta\sigma^\gamma(x^\gamma_n) \right)
+ \textnormal{terms depending on } x_0,\ldots,x_n.
\]
Hence, the Jacobian determinant for the change of variable is
\[
\begin{aligned}
  J^\gamma_{n+1} &= \det \pp{x_{n+1}^\gamma}{x_{n+1}}
= \left( 1+  \frac {\gamma-\gamma^*} {\sigma(x_n)} \delta\sigma^\gamma(x^\gamma_n)  \right)^M
\\
&= 1+  \frac {M(\gamma-\gamma^*)} {\sigma(x_n)} \delta\sigma^\gamma(x^\gamma_n) 
+ \textnormal{higher-order terms in } \gamma-\gamma^*.
\end{aligned}
\]
Temporarily assuming that the change of variables in \Cref{e:changeVar} is valid for \(\gamma\) near \(\gamma^*\), \Cref{e:ephi} becomes
\begin{equation} \label{e:ephi2}
\E{\Phi(X^\gamma_N)} = 
\int_{x_1}\cdots\int_{x_{N}}
\Phi(x^\gamma_N) p(x^\gamma_0,x^\gamma_1)J^\gamma_1\cdots p(x^\gamma_{N-1},x^\gamma_N)J^\gamma_N
dx_1\cdots dx_N,
\end{equation}
Note that here the \(x_n^\gamma\)'s and \(J_n^\gamma\)'s are functions of the \(x_n\)'s, that is, all perturbed quantities are now pulled back to the unperturbed path coordinates.
From now on, unless otherwise noted, we \textit{only} use the \(x_n\)'s as dummy variables when expressing expectations by multiple integrals.

\subsubsection{Taking the derivative}
\label{s:diff}

The main result below is also useful for numerical computations.

\begin{lemma}[discrete-time differentiation]
\label{l:Dt}
\[
\delta \E{\Phi(X^\gamma_N)}
=
\E{ \nabla \Phi(X_N) \cdot v_N
+\Phi(X_N) \sum_{n=0}^{N-1} \frac{\DB_n }{\sigma(X_n)}
\cdot \alpha_n v_n.
}
\]
\end{lemma}

\begin{proof}
Note that when \(\gamma=\gamma^*\), we have \(x_n^\gamma=x_n\), and hence \(J_n=1\).
Differentiating \Cref{e:ephi2} at \(\gamma=\gamma^*\) and applying the Leibniz rule, we obtain
\begin{equation} \label{e:ephi3}
\delta \E{\Phi(X^\gamma_N)}
= 
\E{\delta \Phi(X^\gamma_N) }
+ \sum_{n=0}^{N-1} \E{\Phi(X_N) \left(\frac {\delta p(X^\gamma_n,X^\gamma_{n+1})}{p(X_n,X_{n+1})} + \delta J^\gamma_{n+1}\right)}.
\end{equation}

We now compute each of these terms.
For the first term, since \(\gamma=\gamma^*\),
\[
\delta \Phi(X^\gamma_N)
= \nabla\Phi \cdot \delta X_N
= \nabla\Phi \cdot v_N.
\]

For the second term,
\[
\frac {\delta p(X^\gamma_n,X^\gamma_{n+1})}{p(X_n,X_{n+1})}
= - M \frac{\delta \sigma^\gamma (X^\gamma_n)}{\sigma(X_n)}
- \frac{\DB_n \cdot (v_{n+1} - v_n - \delta F^\gamma(X^\gamma_n) \Delta t)}{\sigma(X_n) \Dt}
+ \frac{\DB_n^2 \, \delta\sigma^\gamma(X^\gamma_n) }{\sigma(X_n)\Dt}.
\]
Here \(\DB_n\) is viewed as a function of \(X_n\) and \(X_{n+1}\), with expression given in \Cref{e:db}, while \(v_n\) is a function of \(\{X_i\}_{0\le i \le n}\); moreover, \(\DB_n^2 := \DB_n\cdot\DB_n\).

For the third term,
\[
\delta J^\gamma_{n+1}
= \frac{M} {\sigma(X_n)} \delta \sigma^\gamma (X^\gamma_n) .
\]

The summand in the sum in \Cref{e:ephi3} is the perturbation applied to the probability kernel at each step.
After cancellation, we obtain
\begin{equation} \label{e:shift}
\frac {\delta p(X^\gamma_n,X^\gamma_{n+1})}{p(X_n,X_{n+1})}
+ \delta J^\gamma_{n+1}
= 
- \frac{\DB_n }{\sigma(X_n) \Dt}
\cdot (v_{n+1} - v_n - \delta F^\gamma(X^\gamma_n) \Dt - \delta\sigma^\gamma(X^\gamma_n) \DB_n).
\end{equation}
By the definition of \(v_n\) in \Cref{e:vDt},
\[
\frac {\delta p(X^\gamma_n,X^\gamma_{n+1})}{p(X_n,X_{n+1})}
+ \delta J^\gamma_{n+1}
= \frac{\DB_n }{\sigma(X_n) \Dt}
\cdot \alpha_n v_n \Dt
= \frac{\DB_n }{\sigma(X_n)}
\cdot \alpha_n v_n.
\]
For the pure path-perturbation method, \(\alpha\equiv0\), so the above term vanishes since nothing is shifted to the kernel.
In our method, however, we shift a small portion of the path perturbation, proportional to \(\alpha\Dt\), and use it to differentiate the kernel.

The need for the temporary change of variables assumption is removed in \Cref{s:removeTrunc}, which defines a truncated path map, differentiates its valid change-of-variable identity, and then lets \(R\to\infty\). This completes the proof without additional assumptions.
\end{proof}

\subsection{Passing to the continuous-time limit}
\label{s:cts}

We now let \(\Dt\to 0\) and pass to the continuous-time setting.
The main idea of the proof is that all bounds in \Cref{ass} are uniform in \(\gamma\), which allows us to prove \(\gamma\)-uniform convergence of the discrete \(\gamma\)-derivatives.
This in turn implies differentiability of \(\E{\Phi(X_T^\gamma)}\) in the SDE limit as \(\Dt\to 0\).

\goldbach*

\begin{proof}
In this proof, every \(\delta\)-quantity carrying a displayed \(\gamma\), including \(\delta f_N(\gamma)\), \(\delta F^\gamma\), and \(\delta\sigma^\gamma\), denotes the derivative evaluated at that displayed \(\gamma\in I\).
Note that the equations for \(X_n\) in \Cref{e:discreteSDE} and for \(v_n\) in \Cref{e:vDt} form the Euler--Maruyama scheme for the SDEs of \(X_t\) and \(v_t\) in the theorem.
We denote by \(\bar X_t^{\gamma,N}\) the continuous-time piecewise-constant process associated with the discrete-time solution \(\{X_n^\gamma\}\) when the time interval \([0,T]\) is partitioned into \(N\) subintervals, namely,
\[
\bar X_t^{\gamma,N} = X_n^\gamma
\qquad \text{for } t\in[n\Delta t,(n+1)\Delta t).
\]
Similarly, for \(\alpha\) we define \(\bar\alpha_t^{N}:=\alpha_{n\Dt}\) for \(t\in[n\Delta t,(n+1)\Delta t)\).
The bounds in \Cref{ass} are uniform in \(\gamma\in I\).

Under \Cref{ass} (A1)--(A3), the joint process \(Y_t^\gamma:=(X_t^\gamma,v_t^\gamma)\) satisfies a \(2M\)-dimensional SDE driven by the same \(M\)-dimensional Brownian motion.
The joint coefficients need not be globally Lipschitz in \((x,v)\), but the \(v\)-equation is linear in \(v\).
For \(|\alpha|,|\widetilde\alpha|\le\alpha^*\), (A1) gives
\[
\begin{aligned}
|G_\alpha^\gamma(x,v)-G_{\widetilde\alpha}^\gamma(y,w)|
&\le (\alpha^*+L)|v-w|+L(1+|w|)|x-y|+|\alpha-\widetilde\alpha||w|,\\
|H^\gamma(x,v)-H^\gamma(y,w)|
&\le L|v-w|+L(1+|w|)|x-y|.
\end{aligned}
\]
Together with (A2)--(A3), these bounds give moments through order eight, uniformly in \(\gamma\), for \((X^\gamma,v^\gamma)\) and their Euler--Maruyama approximations.
Using these bounds and \(|\alpha_t-\bar\alpha_t^N|\le L_\alpha\Delta t\), the usual Burkholder--Davis--Gundy and Gronwall estimates for the coupled Euler errors give the result below for \(p=2,4\); Jensen's inequality gives \(p=1\) \cite{E2019}.
There exists \(C>0\) (depending only on \(x_0\), \(v_0\), \(I\), \(T\), \(L\), \(K\), \(\alpha^*\), and \(L_\alpha\)) such that, for all sufficiently large \(N\),
\begin{equation}\begin{split} \label{e:strongXv}
\sup_{\gamma\in I}\sup_{0\le t\le T}
\E{|X_t^\gamma - \bar X_t^{\gamma,N}|^p
+ |v_t^\gamma-\bar v_t^{\gamma,N}|^p}
&\le C \Dt^{\frac p2},
\quad \textnormal{where} \quad p = 1,2,4.
\end{split}\end{equation}

To prove convergence of \(\E\Phi\), note that by assumption (A5) and \Cref{e:strongXv}, \(\Phi(\bar X_T^{\gamma,N})\to\Phi(X_T^\gamma)\) in \(L^1\) uniformly in \(\gamma\).
That is, if we define \(f_N(\gamma):=\E{\Phi(\bar X_T^{\gamma,N})}\) and \(f(\gamma):=\E{\Phi(X_T^\gamma)}\), then
\begin{equation*}\label{eq:unif-f}
  \sup_{\gamma\in I}|f_N(\gamma)-f(\gamma)| \rightarrow 0,
  \quad \textnormal{as } \Dt\to0.
\end{equation*}

We next prove the uniform convergence of the derivatives.
By \Cref{l:Dt}, the linear response of the time-discretized SDE is
\[
  \delta f_N(\gamma)
  :=\delta \E{\Phi(\bar X_T^{\gamma,N})}
  =\delta \E{\Phi(X_N^{\gamma})}
  =
  \E{ \nabla \Phi(X^\gamma_N) \cdot v_N^\gamma
  +\Phi(X^\gamma_N) \sum_{n=0}^{N-1} \frac{\DB_n }{\sigma^\gamma(X^\gamma_n)}
  \cdot \alpha_n v^\gamma_n},
\]
and we define the limiting functional by
\begin{equation*}\begin{split}
  g(\gamma)
  :=
  \E{
  \nabla\Phi(X_T^\gamma)\cdot v_T^\gamma
  +
  \Phi(X_T^\gamma)\int_0^T
  \frac{\alpha_t\, v_t^\gamma}{\sigma^\gamma(X_t^\gamma)}\cdot dB_t
}.
\end{split}\end{equation*}
By (A5) and \Cref{e:strongXv}, \(\Phi(\bar X_T^{\gamma,N})\to\Phi(X_T^\gamma)\), \(\nabla\Phi(\bar X_T^{\gamma,N})\to\nabla\Phi(X_T^\gamma)\), and \(\bar v_T^{\gamma,N}\to v_T^\gamma\) in \(L^2\), uniformly in \(\gamma\).

We prove the \(L^2\)-convergence of the discrete sum in \(\delta f_N\) by rewriting it as a stochastic integral and applying It\^o's isometry:
\[
\begin{aligned}
&
  \E{\left|
  \sum_{n=0}^{N-1}\frac{\alpha_n\, v_n^{\gamma}}{\sigma^\gamma (X_n^{\gamma})}\Delta B_n
  - \int \frac{\alpha_t\, v_t^{\gamma}}{\sigma^\gamma (X_t^{\gamma})} d B
  \right|^2 }
\\
={}&
  \E{\left|
  \int_0^T 
  \left(
  \frac{\bar \alpha^N _t \,\bar v_t^{\gamma,N}}{\sigma^\gamma(\bar X_t^{\gamma,N})}
  - \frac{\alpha_t\, v_t^{\gamma}}{\sigma^\gamma (X_t^{\gamma})}
  \right) d B_t
  \right|^2 }
\\
={}&
  \E{
  \int_0^T 
    \left|
  \frac{\bar \alpha^N _t \,\bar v_t^{\gamma,N}}{\sigma^\gamma(\bar X_t^{\gamma,N})}
  - \frac{\alpha_t\, v_t^{\gamma}}{\sigma^\gamma (X_t^{\gamma})}
  \right|^2 dt
}.
\end{aligned}
\]
We separate this expression into three parts and apply the uniform bounds on \(\alpha\) and \(\sigma\).
In the next equation, we use the shorthand
\(\bar \alpha_t:= \bar \alpha^N _t\),
\(\bar v_t^{\gamma} := \bar v_t^{\gamma,N}\),
\(\bar \sigma_t^\gamma:=\sigma^\gamma(\bar X_t^{\gamma,N})\), 
and \(\sigma_t^\gamma:=\sigma^\gamma(X_t^{\gamma})\).
\begin{equation*}\begin{split}
&
\E{
  \int_0^T
  \left|
  \frac{\bar \alpha_t\,\bar v_t^{\gamma}}{\bar \sigma_t^\gamma}
  - \frac{\alpha_t\,v_t^{\gamma}}{\sigma_t^\gamma}
  \right|^2 dt
}
\\
={}&
\E{
  \int_0^T
  \left|
  \frac{\bar \alpha_t}{\bar \sigma_t^\gamma}
  (\bar v_t^{\gamma} - v_t^{\gamma})
+
  \frac{v^\gamma_t}{\bar\sigma_t^\gamma}
  ( \bar \alpha_t - \alpha_t)
+
  \alpha_t v^\gamma_t
  \left(
  \frac{1}{\bar \sigma_t^\gamma}
  - \frac{1}{\sigma_t^\gamma}
  \right)
  \right|^2 dt
}
\\
\le{}&
3\E{
  \int_0^T
  \frac{\bar \alpha^2_t}{(\bar \sigma_t^\gamma)^2}
  |\bar v_t^{\gamma} - v_t^{\gamma} |^2 dt
}
+
3\E{
  \int_0^T
  \frac{|v^\gamma_t|^2}{(\bar\sigma_t^\gamma)^2}
  ( \bar \alpha_t - \alpha_t) ^2 dt
}
\\
&\quad+
3\E{
  \int_0^T
  \alpha_t^2 |v^\gamma_t|^2
  \left(
  \frac{1}{\bar \sigma_t^\gamma}
  - \frac{1}{\sigma_t^\gamma}
  \right)^2 dt
}
\\
\le{}&
\frac{3 (\alpha^*)^2}{(\sigma^*)^2}
\E{ \int_0^T |\bar v_t^{\gamma} - v_t^{\gamma} |^2 dt }
+
\frac{3 }{(\sigma^*)^2}
\E{ \int_0^T |v^\gamma_t|^2 ( \bar \alpha_t - \alpha_t) ^2 dt }
\\
&\quad+
3 (\alpha^*)^2
\E{
  \int_0^T
  |v^\gamma_t|^2
  \left(
  \frac{1}{\bar \sigma_t^\gamma}
  - \frac{1}{\sigma_t^\gamma}
  \right)^2 dt
}.
\end{split}\end{equation*}
To see the \(\gamma\)-uniform convergence of the first term to \(0\), use the \(\gamma\)-uniform \(L^2\)-convergence of \(\bar v^\gamma_t\to v^\gamma_t\) in \Cref{e:strongXv}.
For the second term, use the \(\gamma\)-uniform \(L^4\)-bound of \(v^\gamma_t\), and then the \(L^4\)-convergence of \(\bar \alpha_t \to \alpha_t\), which follows from the Lipschitz-path assumption on \(\alpha_t\) in (A3).
For the third term, use the \(L^4\)-bound of \(v^\gamma_t\), together with the Lipschitz condition of \(\sigma(\cdot)\) in (A1) and the \(L^4\)-convergence of \(\bar X_t \to X_t\).

In summary, the estimates above show that \(\delta f_N\to g\) uniformly on \(I\).
Together with the uniform convergence \(f_N\to f\) on \(I\), we conclude that \(f\) is differentiable on \(I\) and \(\delta f=g\).
\end{proof}

Note that adding a constant to \(\Phi\) does not affect the linear response.
Hence, we may freely subtract \(\Phi_T^{avg}:=\E{\Phi(X_T)}\), and the following formula remains valid.
This form is preferable for numerical purposes, since after centralization, that is, after subtracting \(\Phi_T^{avg}\), the estimator (the integrand in the expectation) is smaller.

\begin{corollary}[centralized path-kernel formula]\label{t:centralized}
\[
\delta \E{\Phi(X^\gamma_T)}
= 
\E{ \nabla \Phi(X_T) \cdot v_T
+(\Phi(X_T)-\Phi^{avg}_T) \int_{t=0}^{T} \frac{\alpha_t v_t }{\sigma(X_t)}
\cdot dB.
}
\]
\end{corollary}

\subsection{Formal infinite-time limit}
\label{s:infinite}

We also formally provide a single-path formula for stationary measures, or physical measures.
The statements in this subsection are heuristic, and a rigorous proof is left to later papers.

When we run the SDE for an infinitely long time, if the distribution does not escape to infinity, then the law of \(X^\gamma_t\) typically converges weakly to a stationary measure.
If we also choose \(\alpha_t\) so that \((\alpha_t,X^\gamma_t)\) converges jointly to a stationary process, and so that \(\alpha_t\) is large enough to temper the exponential growth of \(v^\gamma_t\) (see \Cref{s:howtouse}), then \((\alpha_t,X^\gamma_t,v^\gamma_t)\) should also converge to a stationary process.
We assume that the stationary measure of \(X^\gamma\) is unique and denote it by \(\mu^\gamma\). We also assume that the joint stationary process above is unique.

Let \(\mu:=\mu^{\gamma^*}\). For each \(\gamma\), start the \(\gamma\)-dynamics from the fixed baseline distribution \(X_0^\gamma\sim\mu\), with \(v_0=0\). Assuming convergence to \(\mu^\gamma\),
\[
\E[\mu^\gamma]{\Phi(X_0^\gamma)}
= \lim_{T\rightarrow\infty} \E[\mu]{\Phi(X_T^\gamma)}.
\]
The initial measure \(\mu\) on the right-hand side is independent of \(\gamma\), so it is not differentiated. By conditional expectation,
\[
\E[\mu]{\Phi(X_T^\gamma)}
=
\int_{\R^M}\E[x]{\Phi(X_T^\gamma)}\,\mu(dx),
\]
Formally interchanging \(\delta\) with the infinite-time limit in the equality above, and then using dominated convergence to interchange \(\delta\) with the integral over \(\mu\), gives
\[
\delta\E[\mu^\gamma]{\Phi(X_0^\gamma)}
=
\lim_{T\rightarrow\infty}
\int_{\R^M}\delta\E[x]{\Phi(X_T^\gamma)}\,\mu(dx).
\]
Applying \Cref{t:main} with \(v_0=0\), we obtain
\[
\delta \E[\mu^\gamma]{\Phi(X_0^\gamma)}
= \lim_{T\rightarrow\infty}\E[\mu]{
\nabla\Phi(X_T)\cdot v_T
+
(\Phi(X_T)-\Phi^{avg})
\int_{t=0}^{T}
\frac{\alpha_t v_t}{\sigma(X_t)}\cdot dB_t
}.
\]
Here \(\Phi^{avg}:=\E[\mu]{\Phi(X_0)}\).

We further assume decay of correlations, that is, for smooth observables \(f\) and \(g\),
\[
\E[\mu]{ f(X_s, \alpha_s, v_s) g(X_t, \alpha_t, v_t) }
\rightarrow
\E[\mu]{ f(X_0, \alpha_0, v_0) } \E[\mu]{ g(X_0, \alpha_0, v_0)}
\]
rapidly as \(|t-s|\rightarrow\infty\).
Hence, for sufficiently large \(W>0\),
\begin{equation*}\begin{split}
\lim_{T\rightarrow\infty}
\E[\mu]{
\int_{t=0}^{T-W}
  (\Phi(X_T)-\Phi^{avg}) \,
\frac{\alpha_t v_t }{\sigma(X_t)} \cdot dB_t}
\\ \approx
\lim_{T\rightarrow\infty}
\E[\mu]{ (\Phi(X_0)-\Phi^{avg}) }
\,
\E[\mu]{ 
\int_{t=0}^{T-W}
\frac{\alpha_t v_t }{\sigma(X_t)} \cdot dB_t}
\approx 0,
\end{split}\end{equation*}
since both factors in the product vanish: the first is zero by centering, and the second is zero by the basic property of It\^o integration.
After removing the remote-past contribution, the assumed convergence to stationarity allows the process on the remaining terminal window \([T-W,T]\), for each fixed \(W\), to be replaced in the limit \(T\to\infty\) by the joint stationary process.
Hereafter, \((\alpha_t,X_t,v_t)\) denotes the joint stationary process.

In summary,
\[
\delta \E[\mu^\gamma]{\Phi(X_0^\gamma)}
=
\lim_{W\rightarrow\infty}
\E[\mu]{ \nabla \Phi(X_{0}) \cdot v_{0} }
+
\E[\mu]{
\int_{\tau=0}^{W}
  (\Phi(X_W)-\Phi^{avg}) \,
\frac{\alpha_\tau v_\tau }{\sigma(X_\tau)} \cdot dB_\tau}.
\]
Using the ergodic theorem and the convergence to the stationary distribution, the following formal formula computes the above quantity by averaging along a single orbit.

\begin{formula}[formal ergodic path-kernel formula]
\label{t:ergodic}
\[
\delta \E[\mu^\gamma]{\Phi(X^\gamma_0)}
\,\overset{\mathrm{a.s.}}{=}\,
\lim_{W\rightarrow\infty} \lim_{T\rightarrow\infty} 
\frac 1T \int_{t=0}^T \left[
\nabla\Phi(X_t) \cdot v_t
+ \left( \Phi(X_{t+W}) - \Phi^{avg} \right)
\int_{\tau=0}^W  \frac{\alpha_{t+\tau} v_{t+\tau} }{\sigma(X_{t+\tau})} \cdot dB_{t+\tau}
\right] dt.
\]
\end{formula}

For many numerical applications, the decay of correlations is fast, so it is often enough to choose a moderate \(W\) and then a large \(T\).

\section{Algorithm and numerical examples}
\label{s:numeric}

In this section, we first give guidance on the choice of \(\alpha\), then present the algorithm in procedural form, and finally demonstrate it with a numerical example.

\subsection{A rough estimate of the optimal constant \texorpdfstring{$\alpha$}{alpha}}
\label{s:howtouse}

We explain how to choose the schedule \(\alpha_t\).
Note that \(v_t\) depends on \(\alpha_t\), but any choice of \(\alpha_t\) yields the same \(\delta \E{\Phi(X_T^\gamma)}\).
Different choices of \(\alpha_t\), however, can significantly affect the variance of the estimator and hence the numerical performance.

In this subsection, we consider the simplified case where (1) the system is one-dimensional, (2) all coefficients in the dynamics of \(v_t\) are constants, and (3) \(\alpha\) is a constant independent of the filtration.
We derive an approximate expression for the size of the path-kernel estimator $J$ as a function of \(\alpha\).
We further assume that \(J(\alpha)\) is considerably larger than the gradient itself, so that \(J(\alpha)\) is a rough proxy for the standard deviation of the estimator.
Hence, the number of samples needed for a fixed accuracy is roughly \(O(J^2(\alpha))\).
This one-dimensional case also provides guidance for choosing \(\alpha\) in more general situations.

More specifically, let \((B_t)_{t\ge 0}\) be a standard one-dimensional Brownian motion, and consider the scalar It\^o SDE
\begin{equation}\label{eq:sde}
  dv_t 
  = \Big((a_1-\alpha)v_t+a_2\Big)\,dt
  + \Big(a_3 v_t+a_4\Big)\,dB_t,
  \qquad v_0\in\mathbb R,
\end{equation}
where the constant \(a_1 \in \R\) approximates \(\nabla F\),
\(a_2 \ge 0\) approximates \(\delta F^\gamma\),
\(a_3 \in \R\) approximates \(\nabla \sigma\),
and \(a_4 \ge 0\) approximates \(\delta \sigma^\gamma\).

In \Cref{t:main}, we estimate the size of \(v_t\), regardless of \(t\), by its long-time RMS size (when it exists), namely \(\sqrt{m_2^\infty}\), where
\[
  m_2(t) := \E{v_t^2},
  \qquad
  m_2^\infty := \lim_{t\to\infty} m_2(t).
\]
The size of the path-kernel estimator is then estimated by
\begin{equation}\label{eq:Jdef}
  J(\alpha)\;:=\;\big(c_1+c_2\, |\alpha| \,\sqrt{T}\big)\,\sqrt{m_2^\infty},
\end{equation}
where the constant \(c_1 \ge 0\) approximates \(\nabla \Phi\), and \(c_2\ge 0\) approximates \(\Phi/\sigma\).

We now estimate \(m_2^\infty\).
By It\^o's formula,
\[
d(v_t^2)=2v_t\,dv_t+(dv_t)^2.
\]
Since \((dv_t)^2=(a_3 v_t+a_4)^2\,dt\), we obtain
\[
d(v_t^2)
=\Big((2(a_1-\alpha)+a_3^2)v_t^2+2(a_2+a_3a_4)v_t+a_4^2\Big)\,dt
+\Big(2a_3 v_t^2+2a_4 v_t\Big)\,dB_t.
\]
Taking expectation, and writing \(m_1(t):=\mathbb E[v_t]\), we get the closed ODE system
\begin{equation}\begin{split}
  \label{e:m1m2}
  d m_1(t) / dt &= (a_1-\alpha)\,m_1(t) + a_2, \\
  d m_2(t) / dt &=(2(a_1-\alpha)+a_3^2)\,m_2(t)+2(a_2+a_3a_4)\,m_1(t)+a_4^2.
\end{split}\end{equation}
A sufficient condition for \(m_2(t)\) to converge to a finite limit as \(t\to\infty\) is
\begin{equation} \begin{split} \label{e:m2exist}
  \alpha > \alpha_{\mathrm{crit}} := a_1 + \frac 12 a_3^2,
\end{split} \end{equation}
which also implies that \(m_1(t)\) has a finite long-time limit.

Assume \eqref{e:m2exist}. We now derive an estimate of \(J(\alpha)\).
Due to the exponential decay in \Cref{e:m1m2}, we can obtain the long-time limit of \(m_1\) by setting the time derivative to zero, which yields
\[
m_1^\infty:=\lim_{t\to\infty} m_1(t)=\frac{a_2}{\alpha-a_1} \ge 0.
\]
Then the long-time limit of the second moment is
\[
m_2^\infty:=\lim_{t\to\infty} m_2(t)
=
\frac{2(a_2+a_3a_4)\,m_1^\infty + a_4^2}{2(\alpha-a_1)-a_3^2}
=
\frac{a_4^2+\dfrac{2a_2(a_2+a_3a_4)}{\alpha-a_1}}{2(\alpha-a_1)-a_3^2} \ge 0.
\]
Hence, we obtain the explicit relation
\[
J(\alpha)
=
\big(c_1+c_2\, |\alpha| \,\sqrt{T}\big)\,
\sqrt{
\frac{a_4^2+\dfrac{2a_2(a_2+a_3a_4)}{\alpha-a_1}}{2(\alpha-a_1)-a_3^2}
}.
\]

We now describe the asymptotic behavior of \(J(\alpha)\), excluding exceptional parameter choices.
In the unstable setting motivating this estimate, \(a_1\) represents the positive largest Lyapunov exponent of the deterministic part, and hence \(\alpha_{\mathrm{crit}}>0\).
As \(\alpha\downarrow \alpha_{\mathrm{crit}}\), the denominator in \(m_2^\infty\) tends to \(0^+\), and hence
\[
J(\alpha)\sim (\alpha-\alpha_{\mathrm{crit}})^{-1/2}\to\infty
\qquad\text{as }\alpha\downarrow \alpha_{\mathrm{crit}}.
\]
On the other hand, if \(a_4>0\), then as \(\alpha\to+\infty\),
\[
m_2^\infty\sim \frac{a_4^2}{2\alpha}.
\]
If, in addition, \(c_2>0\), then
\[
J(\alpha)\sim \frac{c_2a_4\sqrt{T}}{\sqrt2}\,\alpha^{1/2}\to\infty
\qquad\text{as }\alpha\to+\infty.
\]
Consequently, in this nondegenerate case, there exists an optimizer
\[
\alpha_{\mathrm{opt}}\in(\alpha_{\mathrm{crit}},\infty)
\quad\text{such that}\quad
J(\alpha_{\mathrm{opt}})=\min_{\alpha>\alpha_{\mathrm{crit}}}J(\alpha).
\]
A plot of \(J(\alpha)\) for one particular choice of hyper-parameters is shown in \Cref{f:J_alpha}. This example has \(a_4=c_2=1\), so it satisfies the nondegeneracy conditions used above. As the figure illustrates, \(J(\alpha)\) is relatively flat over a wide range of \(\alpha\) values. In particular, in this example it stays within a small factor of its minimum for \(\alpha\in[3,40]\), which is a very large interval compared to the size of \(\alpha_{\mathrm{crit}}\) (here \(\alpha_{\mathrm{crit}}=1.5\)).

\begin{figure}[ht]
\centering
  \includegraphics[width=0.5\textwidth]{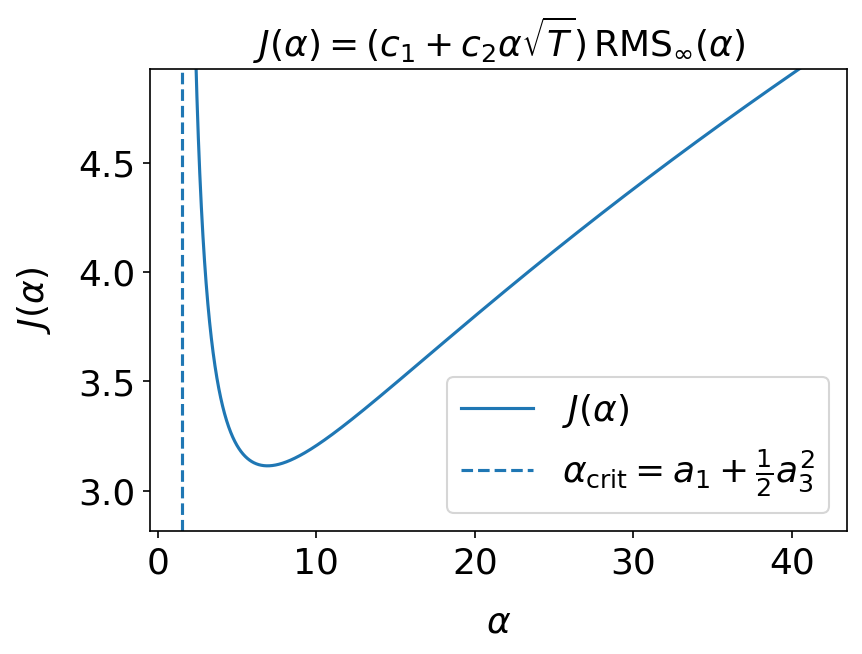}
  \caption{Plot of \(J(\alpha)\) for the one-dimensional example with \(a_1=a_2=a_3=a_4=c_1=c_2=T=1\). Hence, \(\alpha_{\mathrm{crit}}=1.5\).}
  \label{f:J_alpha}
\end{figure}

In practice, we can find an approximately optimal \(\alpha\) (in the sense of minimizing \(J\)) by trial and error.
We may try a few values of \(\alpha\), from small to large, in order to locate \(\alpha_{\mathrm{opt}}\).
This is not too difficult in the nondegenerate case discussed above, since \(J(\alpha)\) has a simple shape: it diverges as \(\alpha\downarrow \alpha_{\mathrm{crit}}\), and also grows monotonically for sufficiently large \(\alpha\).
Moreover, \(J(\alpha)\) is relatively insensitive to the choice of \(\alpha\): once \(\alpha\) is moderately larger than \(\alpha_{\mathrm{opt}}\), \(J(\alpha)\) grows only slowly.
Alternatively, one may estimate \(\alpha_{\mathrm{crit}}\) empirically as the point where the sample average of \(v_t^2\) (over sample paths) changes from exponential growth to stability, and then choose \(\alpha\) to be somewhat larger than \(\alpha_{\mathrm{crit}}\), for example \(5\alpha_{\mathrm{crit}}\).

If numerical efficiency is especially important, then \(\alpha_t\) should be designed more carefully.
Note that \(\alpha_t\) may be a process rather than a fixed constant.
If the system has regions of strong instability, then \(\alpha_t\) should be larger there in order to reduce the growth of \(v_t\).
If the system has regions with small noise, then \(\alpha_t\) should be smaller there, since the \(1/\sigma\) factor in the kernel part of the path-kernel estimator is already large.
In this paper, we only test the simple case where \(\alpha_t\) is a fixed constant.

\subsection{Algorithms}
\label{s:procedure}

We first give \Cref{a:finiteT} for the finite-time setting.
We use the Euler scheme to integrate the SDE, which coincides with our discrete-time formulation in \Cref{l:Dt}.

\begin{algorithm}
\caption{Path-kernel algorithm on the finite time interval \([0,T]\)}\label{a:finiteT}
\begin{algorithmic}
\Require number of sample paths \(L\ge2\), \(\Dt\), \(N=T/\Dt\), \(x_0\), \(v_0\).
\For{\(l = 1, \dots, L\)}
  \For{\(n = 0, \dots, N-1\)}
  \State Independently generate \(\Db_{n}\sim N(0,\Dt I_M)\)
  \State \(\alpha_n \gets \alpha_{n\Dt}\), which may depend on \(\{x_i\}_{i\le n}\)
  \State \(I_n \gets \frac{\Db_n }{\sigma(x_n)} \cdot \alpha_n v_n\)
  \State \(x_{n+1} \gets x_n + F(x_n) \Dt + \sigma(x_n)\Db_n\)
  \State \(v_{n+1} \gets v_n - \alpha_n v_n \Dt 
  + \left(\nabla_{v_n} F + \delta F^\gamma \right)(x_n) \Dt
  + \left(\nabla_{v_n} \sigma + \delta \sigma^\gamma \right) (x_n)\Db_n\)
  \EndFor
  \State \(\Phi_{l}\gets\Phi (x_{N})\)
  \State \(S^1_{l} \gets \nabla \Phi(x_N) \cdot v_N\)
  \State \(S^2_{l} \gets \sum_{n=0}^{N-1} I_n\)
  \Comment No need to store the entire path.
\EndFor
\State \(\Phi_T^{avg} \approx \bar\Phi:= \frac 1L \sum_{l=1}^L \Phi_{l}\)
\State \(\delta \E{\Phi(X^\gamma_T)}
\approx \frac 1L \sum_{l=1}^L S^1_l
+\frac{1}{L-1}\sum_{l=1}^L
\left(\Phi_l-\bar\Phi\right)S^2_l\)
\end{algorithmic}
\end{algorithm}

To explain the normalization in the last step of \Cref{a:finiteT}, note that \(\Phi_j\) is independent of \(S_l^2\) for \(j\ne l\). Hence,
\[
\begin{aligned}
\E{\bar\Phi\sum_{l=1}^L S_l^2}
=\frac{1}{L}\sum_{j=1}^L\sum_{l=1}^L\E{\Phi_jS_l^2}
=\E{\Phi S^2}+(L-1)\E{\Phi}\E{S^2},\\
\E{\frac{1}{L-1}\sum_{l=1}^L(\Phi_l-\bar\Phi)S_l^2}
=\frac{L\E{\Phi S^2}-\E{\bar\Phi\sum_{l=1}^L S_l^2}}{L-1}\\
=\E{\Phi S^2}-\E{\Phi}\E{S^2}
=\E{(\Phi-\E{\Phi})S^2}.
\end{aligned}
\]
Hence the unbiased finite-time estimator is
\[
\frac{1}{L}\sum_{l=1}^L S_l^1
+\frac{1}{L-1}\sum_{l=1}^L
(\Phi_l-\bar\Phi)S_l^2.
\]
For \Cref{a:infinite}, the samples are correlated, but the resulting centering bias is small for large \(N\).

We next give \Cref{a:infinite}, the formal infinite-time version.
Here $\alpha_t$ needs to be stationary, \(W\) is the decorrelation time, and \(T\) is the orbit length; typically \(W \ll T\).

\begin{algorithm}[H]
\caption{Formal path-kernel algorithm for the stationary measure on an infinite time span}\label{a:infinite}
\begin{algorithmic}
\Require \(\Dt, N_W=W/\Dt, N=T/\Dt\).
\State Take any \(x_0\) and \(v_0\).
\For{\(n = 0, \cdots, N_W+N-1\)}
  \State Compute \(\Db_{n}, \alpha_n, I_n, x_{n+1}, v_{n+1}\) as in \Cref{a:finiteT}
  \State \(\Phi_{n}\gets\Phi (x_{n})\)
  \State \(S^1_{n} \gets \nabla\Phi(x_n) \cdot v_n\)
  \State If \(n\ge N_W\), \(S^2_{n} \gets \sum_{m=0}^{N_W-1} I_{n+m-N_W}\)
\EndFor
\State \(\Phi^{avg} \approx \frac 1N \sum_{n=N_W}^{N_W+N-1} \Phi_n\)
\State \(\delta \Phi^{avg} 
\approx  \frac 1N \sum_{n=N_W}^{N_W+N-1} \left( S^1_{n} 
+ \left(\Phi_{n} - \Phi^{avg} \right) S^2_{n} \right)\)
\end{algorithmic}
\end{algorithm}

From a practical point of view, the current version of the path-kernel method is \textit{not} an adjoint algorithm.
Therefore, its cost is proportional to the dimension of \(\gamma\).
The adjoint path-kernel formula, together with its application to data assimilation where one needs to optimize thousands of parameters, is given in \cite{apk}.

\subsection{Numerical example: Lorenz--96 with noise}
\label{s:example}

We demonstrate our algorithm on the Lorenz--96 model \cite{Lorenz96} with noise in \(\R^M\), where the dimension is \(M=40\).
The SDE is
\begin{eqnarray*}
d X^i
= \left( \left(X^{i+1}-X^{i-2}\right) X^{i-1} - X^i + 8 + \gamma^0 \right) dt + (1+\gamma^1)\sigma_0 dB^i
\\
\textnormal{for } i=1, \ldots, M;
\quad
X^\gamma_0 = (1+\gamma^2) [1,\ldots, 1].
\end{eqnarray*}
Here we use the periodic convention \(X^{-1}=X^{M-1}\), \(X^0=X^M\), and \(X^{M+1}=X^1\), where \(X^i\) denotes the \(i\)-th component of the state.
The three parameters \(\gamma^0\), \(\gamma^1\), and \(\gamma^2\) control the drift term, the diffusion term, and the initial condition, respectively.
The observable function \(\Phi\) is chosen as
\[
  \Phi(x) := \frac 1M \sum_{i=1}^M (x^i)^2.
\]
We set \(\sigma_0=1.5\), which corresponds to a relatively small noise level.
The Lorenz--96 example is outside the assumptions of \Cref{t:main}: its quadratic drift does not satisfy the global Lipschitz and linear-growth conditions, and the quadratic observable is not globally Lipschitz. We include it as a numerical illustration beyond the setting covered by the theorem.

The present example is not accessible to the previous linear response methods.
The system is unstable, so the pure path-perturbation method is not effective.
The diffusion coefficient depends on the parameters, so the pure kernel-differentiation method does not apply.
The dynamics depend on the parameters, and we also consider the linear response of stationary measures, so the Bismut formula does not apply in either case.
Finally, the system is non-hyperbolic, so the fast response method does not apply.

First, we compute the parameter derivatives of \(\Phi_T^{avg}:=\E{\Phi(X_T)}\) for \(T=1\).
A typical orbit is shown in \Cref{f:orbit}.
We set \(\Dt=0.002\).
Then, following the discussion in \Cref{s:howtouse}, we choose the constant schedule process
\[
  \alpha_t \equiv 10,
\]
We select this value by trying several choices and retaining the one with the smallest estimator variance.
Thus the instability is sufficiently tempered.
The default values of all \(\gamma^i\)'s are zero.
We compute the derivatives with respect to \(\gamma^0, \gamma^1\), and $\gamma^2$.
Among these, the derivative with respect to \(\gamma^1\) is the main difficulty.
The results are shown in \Cref{f:ga2,f:ga13}; the derivatives computed by the path-kernel method agree well with the observed dependence of \(\Phi_T^{avg}\) on the parameters.

\begin{figure}[ht]
\centering
  \includegraphics[width=0.4\textwidth]{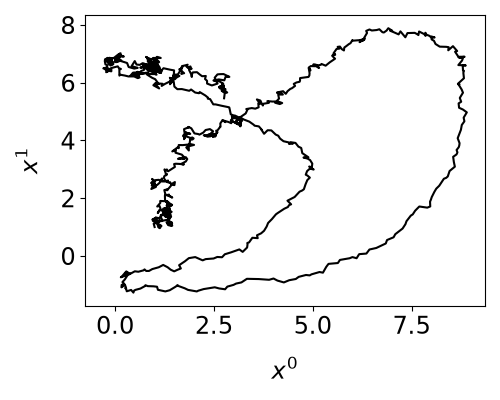}
  \caption{Plot of \(x_t^0\) and \(x_t^1\) along a typical orbit of length \(T=2\).}
  \label{f:orbit}
\end{figure}

\begin{figure}[ht]
\centering
  \includegraphics[width=0.4\textwidth]{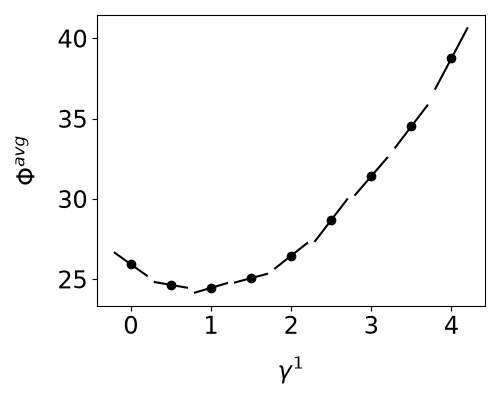}
  \caption{
  \(\Phi_T^{avg}\) and \(\delta \Phi_T^{avg}\) for different values of \(\gamma^1\).
  The dots represent \(\Phi_T^{avg}\), and the short line segments represent \(\delta \Phi_T^{avg}\) computed by the path-kernel algorithm.
  At each parameter value, both are computed from the same ensemble of \(L=1000\) sample paths.
  }
  \label{f:ga2}
\end{figure}

\begin{figure}[ht]
\centering
  \includegraphics[width=0.4\textwidth]{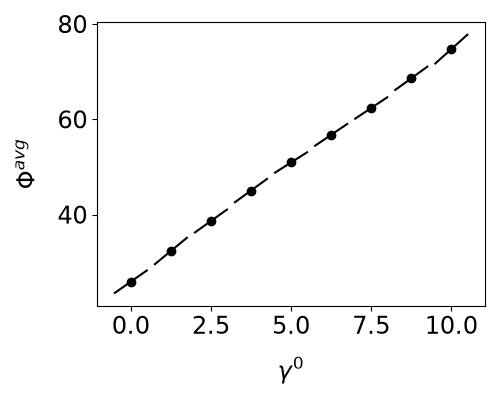}
  \includegraphics[width=0.4\textwidth]{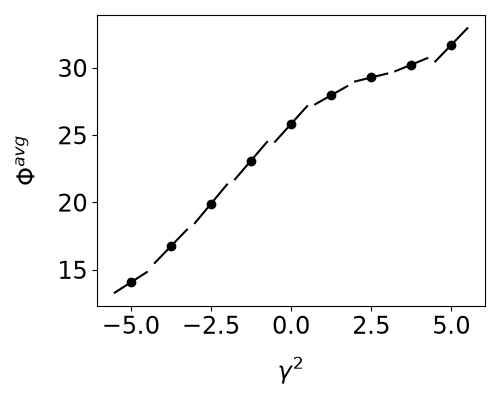}
  \caption{
  \(\Phi_T^{avg}\) and \(\delta \Phi_T^{avg}\) for different values of \(\gamma^0\) and \(\gamma^2\).
  }
  \label{f:ga13}
\end{figure}

Then we compute the linear response with respect to \(\gamma^0\) for the stationary measure.
In \Cref{a:infinite}, we set \(T=1000\) and \(W=1\).
The results are shown in \Cref{f:physical_measure}.
As the figure shows, the algorithm gives an accurate linear response in the noisy case.
Moreover, the linear response with respect to \(\gamma^0\) in the noisy case reasonably reflects the relationship between the observable and the parameter in the deterministic case.

\begin{figure}[ht]
\centering
  \includegraphics[width=0.4\textwidth]{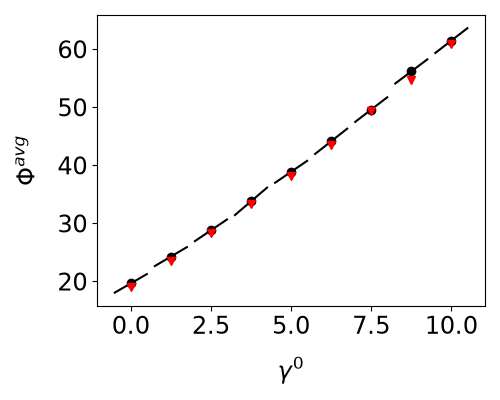}
  \caption{\(\Phi^{avg}\) and \(\delta \Phi^{avg}\) for the stationary measure for different values of \(\gamma^0\).
  The red triangles represent \(\Phi^{avg}\) for the Lorenz--96 system without noise.
  }
  \label{f:physical_measure}
\end{figure}

We also compute the linear response with respect to \(\gamma^1\) for the stationary measure, and compare it with the pure kernel-differentiation method.
By \Cref{s:deker}, the pure kernel-differentiation method is obtained by setting \(\alpha\equiv 1/\Dt\) in the path-kernel formula.
We keep all other settings the same.
\Cref{f:pm_kd_ga1} shows that the path-kernel method performs substantially better than the pure kernel method, which essentially diverges; this divergence is explained in \Cref{s:deker}.
We also tested the pure path-perturbation method, but it overflowed during the computation, so no meaningful comparison is available.

\begin{figure}[ht]
\centering
  \includegraphics[width=0.4\textwidth]{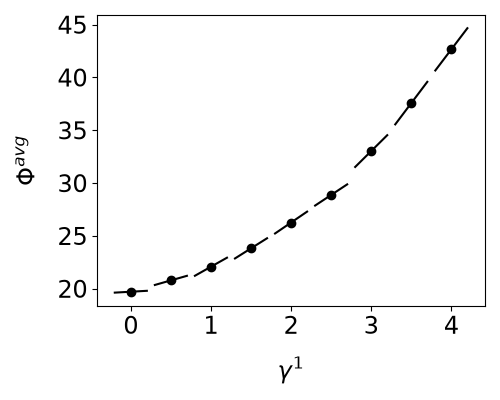}
  \includegraphics[width=0.4\textwidth]{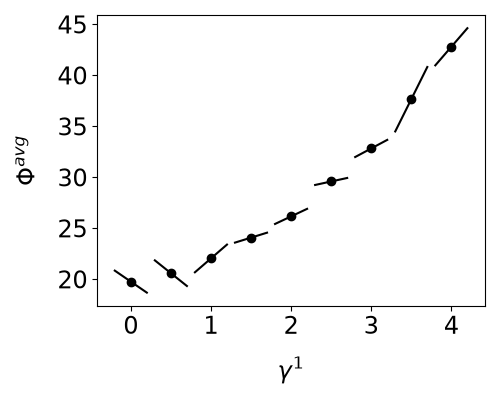}
  \caption{\(\Phi^{avg}\) and \(\delta \Phi^{avg}\) of the stationary measure for different values of \(\gamma^1\).
  Left: derivatives computed by the path-kernel method.
  Right: derivatives computed by the pure kernel-differentiation method.
  }
  \label{f:pm_kd_ga1}
\end{figure}

\section*{Code availability statement}
The code used in this paper is posted at \url{https://github.com/niangxiu/paker}.

\section*{Data availability statement}
There are no other associated data.

\bibliographystyle{abbrv}
{\footnotesize\bibliography{library}}

\begin{appendix}
\crefalias{section}{appendix}
\crefalias{subsection}{appendix}

\section{Removing the Invertibility Assumption for the Path Map}
\label{s:removeTrunc}

We complete the proof of \Cref{l:Dt} without assuming that the change of variables in \Cref{e:changeVar} is globally valid.
We keep \(N\) and \(\Delta t\) fixed.
Fix \(R>0\).
For a given unperturbed path \((x_0,\ldots,x_N)\), set \(v_0^R=v_0\).
Once \(v_n^R\) is known, define
\[
q_n:=\frac{H(x_n,v_n^R)}{\sigma(x_n)},
\qquad
q_n^R:=\operatorname{clip}(q_n,[-R,R]),
\]
Here \(\operatorname{clip}(q,[-R,R])\) leaves \(q\) unchanged inside \([-R,R]\) and otherwise sets it to the nearest endpoint.
Then set
\begin{equation}
\label{e:vDtR}
\begin{aligned}
v_{n+1}^R
={}&v_n^R-\alpha_n v_n^R\Delta t
+\bigl(\nabla_{v_n^R}F(x_n)+\delta F^\gamma(x_n)\bigr)\Delta t
+\sigma(x_n)q_n^R\Db_n.
\end{aligned}
\end{equation}
Although the notation \(q_n\) does not display \(R\), it depends on \(R\) through \(v_n^R\).
This construction is sequential: \(v_n^R\) is used to compute \(q_n^R\), which is then used to compute \(v_{n+1}^R\).

Define the localized path map by
\begin{equation}
\label{e:changeVarR}
x_n^{\gamma,R}:=x_n+(\gamma-\gamma^*)v_n^R,
\qquad 0\le n\le N.
\end{equation}
For the \((n+1)\)-st coordinate change, \(q_n^R\) and all terms in \Cref{e:vDtR} except \(\Db_n\) depend only on \(x_0,\ldots,x_n\).
Combining \Cref{e:db,e:vDtR},
\[
\begin{aligned}
\frac{\partial v_{n+1}^R}{\partial x_{n+1}}
=
\frac{\partial v_{n+1}^R}{\partial \Db_n}
\frac{\partial \Db_n}{\partial x_{n+1}}
=
\left(\sigma(x_n)q_n^R I_M\right)
\left(\frac{1}{\sigma(x_n)}I_M\right)
=q_n^R I_M.
\end{aligned}
\]
Hence, the localized Jacobian factor is
\begin{equation}
\label{e:JR}
J_{n+1}^{\gamma,R}
:=\det\frac{\partial x_{n+1}^{\gamma,R}}{\partial x_{n+1}}
=\left(1+(\gamma-\gamma^*)q_n^R\right)^M.
\end{equation}
When \(|\gamma-\gamma^*|R<1\), this factor is positive, and each sequential coordinate change is affine and bijective.
Therefore, the complete localized path map is bijective.

Changing variables in \Cref{e:ephi} using \Cref{e:changeVarR} gives the valid localized identity
\begin{equation}
\label{e:ephiR}
\E{\Phi(X_N^\gamma)}
=
\int_{x_1}\cdots\int_{x_N}
\Phi(x_N^{\gamma,R})
\prod_{n=0}^{N-1}
p(x_n^{\gamma,R},x_{n+1}^{\gamma,R})J_{n+1}^{\gamma,R}
\,dx_1\cdots dx_N.
\end{equation}
At \(\gamma=\gamma^*\), we have
\[
x_n^{\gamma^*,R}=x_n,
\qquad
J_{n+1}^{\gamma^*,R}=1,
\qquad
\delta x_n^{\gamma,R}=v_n^R.
\]
For fixed \(R\), \Cref{ass} and the moment bounds proved below justify differentiation under the integral by dominated convergence.
Differentiating \Cref{e:ephiR} at \(\gamma=\gamma^*\) and writing the resulting integrals as expectations under the baseline path law, with uppercase notation for the corresponding random variables, gives
\begin{equation}
\label{e:ephiDerivativeR}
\delta\E{\Phi(X_N^\gamma)}
=
\E{\nabla\Phi(X_N)\cdot v_N^R}
+\sum_{n=0}^{N-1}
\E{\Phi(X_N)
\left(
\frac{\delta p(X_n^{\gamma,R},X_{n+1}^{\gamma,R})}
{p(X_n,X_{n+1})}
+\delta J_{n+1}^{\gamma,R}
\right)}.
\end{equation}

The same Gaussian-density calculation as in the proof of \Cref{l:Dt} gives
\[
\begin{aligned}
\frac{\delta p(X_n^{\gamma,R},X_{n+1}^{\gamma,R})}
{p(X_n,X_{n+1})}
={}&-Mq_n
-\frac{\Delta B_n}{\sigma(X_n)\Delta t}
\cdot\left[
v_{n+1}^R-v_n^R
-\bigl(\nabla_{v_n^R}F(X_n)+\delta F^\gamma(X_n)\bigr)\Delta t
\right]
\\
&+\frac{|\Delta B_n|^2}{\Delta t}q_n.
\end{aligned}
\]
By \Cref{e:vDtR}, the expression inside the brackets is
\[
-\alpha_n v_n^R\Delta t
+\sigma(X_n)q_n^R\Delta B_n.
\]
Moreover, \Cref{e:JR} gives \(\delta J_{n+1}^{\gamma,R}=Mq_n^R\).
Consequently,
\begin{equation}
\label{e:shiftR}
\frac{\delta p(X_n^{\gamma,R},X_{n+1}^{\gamma,R})}
{p(X_n,X_{n+1})}
+\delta J_{n+1}^{\gamma,R}
=
\frac{\Delta B_n}{\sigma(X_n)}\cdot\alpha_n v_n^R
+\left(\frac{|\Delta B_n|^2}{\Delta t}-M\right)(q_n-q_n^R).
\end{equation}
Substituting this equality into \Cref{e:ephiDerivativeR} gives
\begin{equation}
\label{e:localizedDerivative}
\begin{aligned}
\delta\E{\Phi(X_N^\gamma)}
={}&\E{\nabla\Phi(X_N)\cdot v_N^R}
\\
&+\sum_{n=0}^{N-1}
\E{\Phi(X_N)\left[
\frac{\Delta B_n}{\sigma(X_n)}\cdot\alpha_n v_n^R
+\left(\frac{|\Delta B_n|^2}{\Delta t}-M\right)(q_n-q_n^R)
\right]}.
\end{aligned}
\end{equation}

\begin{lemma}[Removal of truncation]
\label{l:removeTrunc}
For each fixed \(N\), as \(R\to\infty\),
\[
\max_{0\le n\le N}\E{|v_n^R-v_n|^2}\longrightarrow0
\]
and
\[
\max_{0\le n<N}
\E{
|q_n-q_n^R|^2
+|\sigma(X_n)(q_n-q_n^R)|^2
}
\longrightarrow0.
\]
Consequently, the right-hand side of \Cref{e:localizedDerivative} converges to the right-hand side of the formula in \Cref{l:Dt}.
\end{lemma}

\begin{proof}
\noindent\textbf{Step 1: Uniform moment bounds.}
The Euler scheme for \(X_n\) has finite fourth moments on the fixed time partition.
Moreover, by (A1)--(A3),
\[
\begin{aligned}
\left|
-\alpha_n v+\nabla_vF(X_n)+\delta F^\gamma(X_n)
\right|
+|H(X_n,v)|
\le C(1+|X_n|+|v|).
\end{aligned}
\]
Since clipping does not increase absolute values,
\[
|\sigma(X_n)q_n^R|
\le |\sigma(X_n)q_n|
=|H(X_n,v_n^R)|.
\]
Using this bound in \Cref{e:vDtR}, we condition on the history up to time \(n\) and use the Gaussian fourth moments of \(\Delta B_n\).
Induction over the finitely many time steps then gives
\[
\sup_{R\ge1}\max_{0\le n\le N}\E{|v_n^R|^4}<\infty.
\]
The growth bound for \(H\) and the lower bound \(\sigma\ge\sigma^*\) then give
\begin{equation}
\label{e:qRmoment}
\sup_{R\ge1}\max_{0\le n<N}
\E{
|H(X_n,v_n^R)|^4+|q_n|^4
}<\infty.
\end{equation}
The same argument without clipping gives \(\max_{0\le n\le N}\E{|v_n|^4}<\infty\).

\medskip
\noindent\textbf{Step 2: Vanishing clipping residuals.}
By the definition of clipping,
\[
|q_n-q_n^R|
\le |q_n|\mathbf 1_{\{|q_n|>R\}}.
\]
On \(\{|q_n|>R\}\), we have \(1\le |q_n|^2/R^2\).
Therefore, \Cref{e:qRmoment} gives
\[
\E{|q_n-q_n^R|^2}
\le \frac{1}{R^2}\E{|q_n|^4}
\longrightarrow0,
\]
uniformly over the finitely many values of \(n\).
Similarly,
\[
|\sigma(X_n)(q_n-q_n^R)|
\le |H(X_n,v_n^R)|\mathbf 1_{\{|q_n|>R\}}.
\]
Since \(|q_n|>R\) and \(\sigma(X_n)\ge\sigma^*\) imply
\(|H(X_n,v_n^R)|>R\sigma^*\), we obtain
\[
\begin{aligned}
\E{|\sigma(X_n)(q_n-q_n^R)|^2}
&\le
\E{
|H(X_n,v_n^R)|^2
\mathbf 1_{\{|H(X_n,v_n^R)|>R\sigma^*\}}
}
\\
&\le
\frac{1}{R^2(\sigma^*)^2}
\E{|H(X_n,v_n^R)|^4}
\longrightarrow0.
\end{aligned}
\]

\medskip
\noindent\textbf{Step 3: Convergence to the untruncated formula.}
Let \(e_n^R:=v_n^R-v_n\).
Since \(\sigma(X_n)q_n=H(X_n,v_n^R)\), subtracting the untruncated recursion in \Cref{e:vDt} from \Cref{e:vDtR} gives
\[
\begin{aligned}
e_{n+1}^R
={}&e_n^R-\alpha_n e_n^R\Delta t
+\nabla_{e_n^R}F(X_n)\Delta t
\\
&+
\sigma(X_n)(q_n^R-q_n) \Delta B_n
+
\left[
H(X_n,v_n^R)-H(X_n,v_n)
\right]\Delta B_n.
\end{aligned}
\]
The last difference inside the brackets is bounded by \(L|e_n^R|\).
Taking expectations after conditioning on the history up to time \(n\) therefore gives
\[
\E{|e_{n+1}^R|^2}
\le
(1+C\Delta t)\E{|e_n^R|^2}
+C\Delta t\E{|\sigma(X_n)(q_n-q_n^R)|^2},
\]
where \(C\) is independent of \(R\).
Since \(e_0^R=0\), Step 2 and a finite iteration give
\[
\max_{0\le n\le N}\E{|v_n^R-v_n|^2}\longrightarrow0.
\]

It remains to pass to the limit in the localized differentiated expectation.
The moment bounds above, the Gaussian moments of \(\Delta B_n\), the lower bound for \(\sigma\), and the Cauchy--Schwarz inequality give, for each \(n\),
\[
\E{
\left|
\nabla\Phi(X_N)\cdot(v_N^R-v_N)
\right|
}
\longrightarrow0,
\]
\[
\E{
\left|
\Phi(X_N)
\frac{\Delta B_n}{\sigma(X_n)}
\cdot\alpha_n(v_n^R-v_n)
\right|
}
\longrightarrow0,
\]
and
\[
\E{
\left|
\Phi(X_N)
\left(\frac{|\Delta B_n|^2}{\Delta t}-M\right)
(q_n-q_n^R)
\right|
}
\longrightarrow0.
\]
There are only finitely many time steps, so these convergences prove the final statement of the lemma.
\end{proof}

For every fixed \(R\), \Cref{e:localizedDerivative} follows from the valid localized change of variables, and its left-hand side is the original derivative, independent of \(R\).
Taking \(R\to\infty\) on its right-hand side and applying \Cref{l:removeTrunc} gives
\[
\delta\E{\Phi(X_N^\gamma)}
=
\E{\nabla\Phi(X_N)\cdot v_N
+\Phi(X_N)\sum_{n=0}^{N-1}
\frac{\Delta B_n}{\sigma(X_n)}\cdot\alpha_n v_n}.
\]
This is the formula in \Cref{l:Dt}; hence its proof does not require the path map in \Cref{e:changeVar} to be invertible.

\section{Alternative derivation using the Girsanov theorem}
\label{s:proof2}

Now that we know the path-kernel formula, we can give a simpler and more formal derivation by decomposing the pure path-perturbation \(u\) into \(v\) and an auxiliary process \(w\), and then integrating by parts in \(w\).
This derivation is somewhat retrospective: without already knowing \(v\), it may be difficult to discover this proof.
It also gives a less unified picture than our main proof, since the way we handle \(w\) and \(v\) must be interpreted from different viewpoints, whereas the main proof provides a single explicit rule for pathwise comparison.

Let \(u_t\) be the solution of the (undamped) pure path-perturbation equation
\begin{equation} \label{e:u}
du = 
\left(\nabla_u F(X) + \delta F^\gamma (X) \right) dt
+ \left(d\sigma(X)u + \delta \sigma^\gamma (X)\right)dB,
\quad
u_0 = v_0.
\end{equation}
This is the governing equation for the path perturbation of \(\{x_t\}\) caused by changing \(\gamma\), under a fixed Brownian motion \(\{B_t\}\).
The corresponding pure path-perturbation formula for the linear response is
\[
\delta \E{\Phi(X^\gamma_T)} = 
\E{ \nabla\Phi(X_T) \cdot u_T}.
\]

Define an auxiliary process \(w_t\) by
\begin{equation} \label{e:w}
dw
= \alpha_t v\,dt + \nabla_w F(X)\,dt + d\sigma(X)w\,dB,
\quad
w_0 = 0.
\end{equation}
Recall that \(v\) has the same initial condition as \(u\), and
\[
dv 
= - \alpha_t v\,dt 
+ \left(\nabla_v F(X) + \delta F^\gamma (X) \right) dt
+ \left(d\sigma(X)v + \delta \sigma^\gamma (X)\right)dB.
\]
Hence, we obtain the decomposition of \(u\), and correspondingly of the linear response,
\begin{eqnarray*}
  u = v + w,
  \quad \textnormal{and} \quad 
\delta \E{\Phi(X^\gamma_T)} = 
\E{ \nabla\Phi(X_T) \cdot v_T}
+ \E{ \nabla\Phi(X_T) \cdot w_T}.
\end{eqnarray*}

Then we use the Girsanov transform to integrate by parts in \(w\).
More specifically, consider another parameterized SDE with parameter \(\gamma'\), which perturbs only the drift:
\begin{equation*}\begin{split}
  dX_t^{\gamma'}
  = F(X_t^{\gamma'})\,dt
  + \gamma' \alpha_t v_t\,dt
  + \sigma(X_t^{\gamma'})\,dB,
  \quad \textnormal{with} \quad 
  X_0^{\gamma'} =  x_0.
\end{split}\end{equation*}
Then, under each fixed Brownian path \(\{B_t\}\), the pure path perturbation is \(\delta' X_t^{\gamma'}=w_t\), where \(\delta' := \partial/\partial \gamma'\), since it satisfies the same governing equation as in \Cref{e:w}.
Note that \(v_t\) has bounded moments, and \Cref{t:CMG} gives
\begin{equation*}\begin{split}
  \E{\nabla\Phi(X_T) \cdot w_T}
  =
  \delta' \E{\Phi(X_T^{\gamma'})} 
  =
  \E{\Phi(X_T) \int_{t=0}^{T} \frac{\alpha_t v_t }{\sigma(X_t)}\cdot dB_t}.
\end{split}\end{equation*}
Substituting this into the previous decomposition yields \Cref{t:main}.

We remark that, although the moments of \(v_t\) do not grow exponentially, the unstable growth in \(u_t\) has not disappeared: it remains in \(w_t=u_t-v_t\), which may still grow without bound.
This causes no problem here because the Girsanov transformation depends only on the source term \(\alpha_t v_t\,dt\), which is controlled in moments, not on \(w_t\) itself.

\section{Degeneration to some well-known cases}
\label{s:degen}

\subsection{ Degeneration to the pure path-perturbation formula }
\label{s:depath}

Consider the special case where we set
\[
  \alpha_t \equiv 0.
\]
Then no path perturbation is transferred to the probability kernel.
Hence, \(v_t\) reduces to \(u_t\), and our linear response formula reduces to the pure path-perturbation formula in \Cref{e:u}.

When the path perturbation is unstable, \(u_t\) grows exponentially fast, so the estimator becomes too large. When its \(L^2\) norm is much larger than the gradient itself, this \(L^2\) norm is roughly the standard deviation, and computing the expectation then requires too many samples.
In contrast, the schedule function \(\alpha_t\) in our path-kernel formula introduces a damping term in \(v_t\), which tempers the gradient explosion.

Intuitively, the pure path-perturbation method first fixes the Brownian motion \(\{B_t\}_{t\ge0}\), so the SDE becomes an ODE under that Brownian path, and we then consider how the trajectory \(\{X_t\}_{t\ge0}\) changes when \(\gamma\) is varied.
The resulting change in the value of \(\Phi\) at time \(T\) is then measured by \(\nabla\Phi(X_T)\cdot u_T\), and the linear response is obtained by averaging this quantity over many Brownian paths.
In contrast, as we will see later, the pure kernel-differentiation method fixes the path \(\{X_t\}_{t\ge0}\), and then asks how the Brownian motion \(\{B_t\}_{t\ge0}\) must be modified in order to produce the same path.
As a result, the probability density of obtaining that path changes, which induces a change in the averaged observable.

\subsection{Degeneration to the pure kernel-differentiation formula}
\label{s:deker}

First, this subsection reviews the pure kernel formula under fixed $\sigma$ and how it is derived from the Girsanov theorem.
Then we give a proof of this degenerate case based on modifying the proof of \Cref{l:Dt}.
We then give a consistency check showing that the pure kernel formula can also be obtained by taking a very large schedule \(\alpha_t\).
Together with \Cref{s:depath}, we see that the path-kernel formula is closer to the path method for small \(\alpha\), and closer to the kernel method for large \(\alpha\).

Consider the special case where \(\sigma(x)\) is independent of \(\gamma\), and \(v_0=0\) (so there is no perturbation in the initial condition).
That is,
\[
dX^\gamma_t
= F^\gamma(X^\gamma_t)\,dt 
+ \sigma(X^\gamma_t)\,dB,
\quad
X_0^\gamma = x_0.
\]
Let \(\Phi_T(X^\gamma):=\Phi(X^\gamma_T)\), so \(\Phi_T\) is a function on paths.
\footnote{
Note that the path-kernel result also applies to functionals on continuous paths, not just functions evaluated at a particular time, but we keep the presentation simple in this paper.
}
The Girsanov theorem gives
\[
\E{\Phi_T(X^\gamma)}
= 
\E{\Phi_T(X) 
\exp{-\int_0^T \frac{F-F^\gamma}{\sigma} (X_t)\,dB_t 
-\frac 12 \int_0^T \frac{(F-F^\gamma)^2}{\sigma^2} (X_t)\,dt 
}}
\]
Differentiating with respect to \(\gamma\), we immediately obtain
\begin{theorem}[pure kernel formula under fixed $\sigma$] \label{t:CMG}
\[
\delta \E{\Phi_T(X^\gamma)}
= 
\E{ \Phi_T(X) \int_0^T \frac{1}{\sigma}\,\delta F \cdot dB_t }.
\]
\end{theorem}

\begin{proof}[Proof of \Cref{t:CMG} by modifying our proof]
In the change of variables in \Cref{e:changeVar}, we set \(v_n \equiv 0\), so \(x^\gamma_n\equiv x_n\) for all $n$.
That is, when changing across different values of \(\gamma\), we fix the path \(\{x_n\}_{n=0}^N\).
As a result, among the three factors contributing to the linear response listed in \Cref{s:intui}, the path-perturbation factor and the volume-change factor vanish.
In \Cref{e:ephi3}, this means that the first and last terms are zero.

We only need to compute how the overall probability density changes for each path.
To do this, we first write the kernel within each small time step:
\[
p(x^\gamma_n,x^\gamma_{n+1})
= (2\pi\Delta t)^{-\frac M2}
\sigma^{-M}(x_n)
\exp{-\frac{|x_{n+1} - x_n - F^\gamma(x_n) \Delta t|^2 }{2\sigma^2(x_n)\Dt}}.
\]
Hence, the change in the kernel is
\[
\frac {\delta p(X^\gamma_n,X^\gamma_{n+1})}{p(X_n,X_{n+1})}
=
\frac{\DB_n \cdot ( \delta F^\gamma(X^\gamma_n) \Delta t)}{\sigma(X_n) \Dt}.
\]
Substituting this into \Cref{e:ephi3}, we get
\[
\delta \E{\Phi(X^\gamma_N)}
= 
\sum_{n=0}^{N-1} \E{\Phi(X_N) \frac {\delta p(X^\gamma_n,X^\gamma_{n+1})}{p(X_n,X_{n+1})} 
}
= 
\sum_{n=0}^{N-1} \E{\Phi(X_N) \frac{\DB_n \cdot \delta F^\gamma(X^\gamma_n )}{\sigma(X_n)} 
}.
\]
We may then pass to the continuous-time limit as in \Cref{s:cts}.
\end{proof}

This simple proof also explains why the pure kernel method does not work when \(\sigma^\gamma(x)\) depends on \(\gamma\).
Within each time step, the Gaussian kernel has scale \(O(\sqrt{\Delta t})\).
In the simplified case above, the perturbation moves the center of the kernel by \(\delta F^\gamma(X_n)\Dt\), which is of order \(O(\Delta t)\), and is therefore a relatively small perturbation of the kernel.
However, when \(\sigma^\gamma(x)\) depends on \(\gamma\), the perturbation also changes the scale of the kernel by \(O(\DB)\).
This change is of the same size as the kernel itself, producing an \(O(1)\) contribution to the linear response in one step; after summing over many steps, the result diverges.
The pictorial explanation in \cite{Ni_kd} may help the reader understand this more clearly.

The pure kernel method can still handle perturbations of the diffusion in discrete-time random dynamical systems, since in that setting \(\Dt\) does not go to zero.
However, even there, the path-kernel method performs better than the pure kernel method, because it reduces the size of the integrand and hence the number of required samples.
We can also easily extend the path-kernel method to discrete-time systems with non-Gaussian noise, as in \cite{Ni_kd}.
For continuous-time SDEs, the path-kernel method is not merely beneficial; it is necessary.

Finally, we formally show that the pure kernel formula is obtained from the path-kernel formula when $\alpha_t$ takes very large values.

\begin{proof}[Formal consistency check of \Cref{t:CMG} as a special case of \Cref{l:Dt}]

In \Cref{l:Dt}, set
\[
  \alpha_n \equiv \frac 1 {\Dt}.
\]
Consequently, the governing equation for \(v_n\), namely \Cref{e:vDt}, becomes
\[
 v_{n+1} 
= \nabla_{v_n}F(x_n)\Dt + \delta F^\gamma(x_n)\Dt + d\sigma(X_n) v_n\Delta b_n,
\quad \textnormal{with} \quad 
v_0 = 0.
\]
By linear superposition, we obtain the expression
\[
 v_n = \sum_{m=0}^{n-1} 
 (\nabla F \Dt + d\sigma \Delta b)^{n-m-1}
 \delta F^\gamma(x_m)\Dt.
\]
Note that \(v_n\) is of order \(O(\Dt)\) for all \(n\), since the propagation factor \((\nabla F \Dt + d\sigma \Delta b)\) is of small order.

Since \(v_N\rightarrow 0\) as \(\Dt\rightarrow0\), we may ignore the first term in \Cref{l:Dt}, so that
\begin{eqnarray*}
\delta \E{\Phi(X^\gamma_T)}
= 
\E{ \nabla\Phi(X_N) \cdot v_N
+\Phi(X_N) \sum_{n=0}^{N-1} \frac{\DB_n }{\sigma(X_n)}
\cdot \alpha_n v_n}
\\
\approx \E{
\Phi(X_N) \sum_{n=0}^{N-1}  \frac{\DB_n }{\sigma(X_n)} \cdot \sum_{m=0}^{n-1} 
 (\nabla F \Delta t + d\sigma \DB)^{n-m-1}
 \delta F^\gamma(X_m) }.
\end{eqnarray*}
In the sum over \(m\), the term with \(m=n-1\) is of order \(O(1)\), so it must be retained.
The terms with \(m\le n-3\) are at least of order \(O(\Dt)\), so we neglect them.
Thus the only remaining case to examine is the term
\[
d\sigma(X_{n-1}) \Delta B_{n-1}\,\delta F^\gamma(X_{n-2}),
\]
which arises from the case \(m=n-2\).

We now estimate the contribution of this remaining term:
\[
T:=\E{ \Phi(X_N)  \frac{\DB_n }{\sigma(X_n)} \cdot d\sigma(X_{n-1}) \Delta B_{n-1}  \delta F^\gamma(X_{n-2})}.
\]
At first glance, this looks like a cross term with mean zero, but we must take into account the extra factor \(\Phi(X_N)\).
Set
\[
 f_n:= \frac{d\sigma(X_{n-1}) \delta F^\gamma(X_{n-2})}{\sigma(X_n)} \in \cF_n.
\]
Moreover, by a Taylor expansion with respect to \(\Db_n\),
\[
 \E{\Phi(X_N)\mid\cF_{n+1}} 
 = \E{\Phi(X_N)\mid\cF_n,\DB_n=\Delta b_n} 
 = f_0 + f_0'\Delta b_n + \frac 12 f_0''\Delta b_n^2 + \cdots ,
\]
where \(f_0, f_0', f_0''\in \cF_n\).
Taking the conditional expectation over \(\Delta B_n\), given \(\cF_n\), we obtain
\[
\E{ \Phi(X_N) \mid\cF_{n}}
= f_0 + f_0'' O(\Delta t)
+\cdots .
\]
Subtracting the two expressions gives
\[
\E{ \Phi(X_N) \mid\cF_{n+1}}
=
\E{ \Phi(X_N) \mid\cF_{n}} + O (\DB_n) .
\]
Hence,
\begin{eqnarray*}
T
&=&\E{\E{ \Phi(X_N) f_n \DB_n \Delta B_{n-1} \mid\cF_{n+1}}}
=\E{f_n \DB_n \Delta B_{n-1}\E{ \Phi(X_N) \mid\cF_{n+1}}}
\\
&=&\E{f_n \DB_n \Delta B_{n-1} \left(\E{ \Phi(X_N) \mid\cF_{n}} + O (\DB_n)\right) }.
\end{eqnarray*}
The second term is of order \(O(\DB^3)\), so we neglect it.
Let
\[
g_n:=f_n  \Delta B_{n-1}  \E{ \Phi(X_N) \mid\cF_{n}}\in\cF_n.
\]
Then
\[
T
\approx \E{ g_n \DB_n }
= \E{\E{ g_n \DB_n \mid\cF_n}}
= \E{g_n\E{  \DB_n \mid\cF_n}}
=0 .
\]

We are therefore left with
\[
\delta \E{\Phi(X^\gamma_T)}
\approx \E{
\Phi(X_N) \sum_{n=0}^{N-1}\frac{\DB_n }{\sigma(X_n)} \cdot  \delta F^\gamma(X_{n-1}) }.
\]
Replacing \(\delta F^\gamma(X_{n-1})\) by \(\delta F^\gamma(X_n)\) introduces an error term of order \(\DB_{n-1}\), which produces a cross term \(\DB_n\DB_{n-1}\).
By the above estimate, this error vanishes as \(\Dt\rightarrow0\).
Hence,
\[
\delta \E{\Phi(X^\gamma_T)}
\approx \E{
\Phi(X_N) \sum_{n=0}^{N-1}\frac{\DB_n }{\sigma(X_n)} \cdot  \delta F^\gamma(X_{n}) }.
\]
We may then pass to the continuous-time limit.
\end{proof}

\subsection{Degeneration to pure diffusion with parameterized noise}
\label{s:degauss}

We use a simple case to illustrate why the path-perturbation idea can help the pure kernel-differentiation method for perturbations in the diffusion coefficient; this is why we combine the two.
Consider the special case where \(F(\cdot)\equiv 0\), \(x_0=0\), \(v_0=0\), and \(\sigma^\gamma(x)=1+\gamma\), which is independent of \(x\).
In this example, take \(\gamma^*=0\).
The SDE becomes \(dX_t^\gamma=(1+\gamma)\,dB\), so the solution is
\[
  X^\gamma_t = (1+\gamma)B_t.
\]
The time-discrete SDE becomes
\[
  X^\gamma_N = (1+\gamma)\sum_{n=0}^{N-1} \DB_n.
\]

First, consider the pure kernel method. In the discrete-time derivation, set
\[
  \alpha_n \equiv \frac 1 {\Dt}.
\]
Consequently, the governing equation for \(v_n\), namely \Cref{e:vDt}, becomes
\[
 v_{n+1} = \DB_n,\quad
 v_0 = 0.
\]
By \Cref{l:Dt}, the linear response is then given by the pure kernel-differentiation formula
\begin{equation} \label{e:Gaus}
\delta \E{\Phi(X^\gamma_N)} = \E{\nabla\Phi(X_N)\cdot \DB_{N-1}} +
\E{\Phi(X_N) \sum_{n=1}^{N-1} \frac{ \DB_n \cdot \DB_{n-1}  }{\Dt}
}.
\end{equation}
If we were to compute this numerically, we would be summing \(N\) terms of order \(O(1)\), producing a large estimator that requires many sample paths to average, and diverges as \(\Dt\to 0\).

On the other hand, if we set
\[
  \alpha_t \equiv 0,
\]
then \(v_{n+1}=v_n+\DB_n\), and hence
\[
v_N = B_N.
\]
The linear response is now
\[
\delta \E{\Phi(X^\gamma_N)}
= \E{ \nabla\Phi(X_N) \cdot B_N } .
\]
This is precisely the pure path-perturbation formula.

Intuitively, if we do not let the perturbation from the dynamics hit the kernel immediately, but instead allow it to propagate along the path for some time, then the \(\DB\)-terms coming from different time steps largely cancel each other, and the path perturbation remains bounded.
This prevents the integrand from becoming too large.
In our path-kernel method, we let most of the perturbation propagate, and only use a small portion to differentiate the kernel at each step.
As a result, the \(\DB\)-terms also largely cancel each other.

A careful reader may wonder whether the two linear response formulas obtained from the two extreme choices of schedule are actually equivalent.
They are indeed equivalent, but they behave very differently analytically and numerically.
As a consistency check, in \Cref{e:Gaus} we use the Brownian increment as the dummy variable, whose density is
\[
p(\Db) = (2\pi\Delta t)^{-\frac M2} \exp{-\frac{|\Db|^2 }{2\Dt}}.
\]
Hence,
\[
  \nabla p(\Db) = -\frac{\Db }{\Dt}\, p(\Db) .
\]
After reordering the integrations so that the integration in \(\Db_n\) comes first, we write the main term in \Cref{e:Gaus} as
\begin{eqnarray*}
&\E{\frac{\Phi(X_N)}{\Dt} \DB_n \cdot \DB_{n-1}}
\\ =&
\int_{\Db_0}\cdots\int_{\Db_{N-1}}\Phi(x_N) \frac{\Db_n \cdot \Db_{n-1}}{\Dt}  \,
p(\Db_{N-1}) d\Db_{N-1}
\cdots
p(\Db_0)d\Db_0 
\\ =&
-\int \cdots \int_{\Db_{N-1}} \left(\int_{\Db_{n}}\Phi(x_N) \nabla p( \Db_{n})\,d\Db_n \right)\cdot \Db_{n-1} \, p(\Db_{N-1})
d\Db_{N-1}\cdots .
\end{eqnarray*}
Now integrate by parts in \(\Db_n\).
Note that \(\Phi(x_N) = \Phi(\Db_0+\cdots+\Db_{N-1})\), so
\[
\frac{\partial \Phi(x_N)}{\partial \Db_n} = \nabla \Phi(x_N).
\]
Hence,
\begin{equation*}\begin{split}
&
\E{\frac{\Phi(X_N)}{\Dt} \DB_n \cdot \DB_{n-1}}
\\
={}& \int \cdots\int_{\Db_{N-1}} 
\left(\int_{\Db_{n}}\nabla\Phi(x_N)\, p( \Db_{n})\,d\Db_n \right)
\cdot \Db_{n-1}
p(\Db_{N-1})\,d\Db_{N-1}
\cdots
\\ ={}&
\E{\nabla \Phi(X_N)\cdot \DB_{n-1}}.
\end{split}\end{equation*}
Therefore, \Cref{e:Gaus} becomes
\[
\delta \E{\Phi(X^\gamma_N)} 
= \E{\nabla \Phi(X_N) \cdot \DB_{N-1}} 
+ \sum_{n=1}^{N-1} \E{\nabla \Phi(X_N) \cdot \DB_{n-1}}
= \E{\nabla \Phi(X_N)\cdot B_{N}},
\]
which is exactly the pure path-perturbation formula.

\subsection{Degeneration to the Bismut formula}
\label{s:debis}

We explain how our main result degenerates to the Bismut formula.
Before this paper, the relation between linear response and the Bismut formula was not identified; indeed, some authors worked on both topics without pointing out the possible connection \cite{HM06,HaMa10}.
To show this degeneration, consider the special case where \(F(x)\) and \(\sigma(x)\) do not depend on \(\gamma\), and only the initial condition \(X^\gamma_0 = x_0 +(\gamma-\gamma^*) v_0\) depends on \(\gamma\).
To recover the Bismut formula, first let \(u\) be the pure path perturbation:
\[
du = \nabla_u F(X)\,dt + d\sigma(X)u\,dB,
\quad
u_0 = v_0.
\]
Let \(\beta_t = (T-t)/T\); note that \(d \beta_t = -\frac 1T dt\).

In our notation, we set \(v=\beta_t u\), so \(v\) solves
\begin{eqnarray*}
dv 
&=& u\, d \beta_t + \beta_t\, d u
= \frac {v} {\beta_t }\, d \beta_t + \nabla_{\beta_t u} F(X)\, dt + (d\sigma(X)\beta_t u)\, dB
\\
&=& - \frac {1} {T-t} v\, dt 
+ \nabla_{v} F(X)\, dt 
+ (d\sigma(X) v)\, dB ,
\end{eqnarray*}
with initial condition \(v_0\).
Thus our schedule function is
\[
\alpha_t = \frac 1{T-t}.
\]
Hence, \(\alpha_t v_t = \alpha_t \beta_t u_t = u_t/T\).
Also note that \(v_T=0\) by definition of \(v\).
With these observations, our main theorem reduces to the Bismut formula:
\[
\delta \E{\Phi(X^\gamma_T)}
= \E{\Phi(X_T) \int_{t=0}^{T} \frac{\alpha_t v_t }{\sigma(X_t)} \cdot dB}
=  \frac{1 }{T}\E{\Phi(X_T) \int_{t=0}^{T} \frac{u_t }{\sigma(X_t)} \cdot dB}
.
\]
A benefit of this special choice of schedule function \(\alpha_t\) is that we no longer need differentiability of \(\Phi\), since \(v_T=0\).

The downside of using this special schedule is that it tends to damp the instability \textit{un}evenly: it applies too little damping near the beginning of the time interval.
As a result, \(u_t\) can grow too large, and computing the expectation then requires many samples.
In contrast, we choose \(\alpha_t\) so as to reduce the growth rate of \(v_t\) more uniformly.
This contrast is even sharper when computing the linear response of stationary measures, since \(u_t\) may grow without bound over an infinite time interval in an unstable system.

Previous proofs of the Bismut formula use Malliavin integration by parts, martingale arguments, or a decomposition of the path perturbation into state and noise perturbations \cite{Bismut84,EL94,PW19bismut,HM06}.
In contrast, our main proof uses a single state-path change of variables, without first decomposing the tangent or constructing a compensating noise control. This gives a unified picture and is better suited to further generalizations, such as incorporating the divergence method.

\end{appendix}

\end{document}